\documentclass[11pt]{article}
\usepackage{amssymb,amsmath,amsfonts}
\usepackage{stmaryrd}
\usepackage{color}
\usepackage[latin1]{inputenc}
\usepackage{multirow}
\usepackage{subfigure}
\usepackage{enumerate}
\usepackage{epsfig}
\usepackage{epic}
\usepackage{pstricks}
\usepackage{pst-node}
\usepackage{pstricks-add}

\def\BBox{\kern  -0.2cm\hbox{\vrule width 0.2cm height 0.2cm}}

\newtheorem{lemma}{Lemma}[section]
\newtheorem{theorem}{Theorem}[section]
\newtheorem{definition}{Definition}[section]

\newtheorem{remark}{Remark}[section]

\newtheorem{keyword}{Keyword}

\title{New families of small regular graphs of girth 5}

\author{E. Abajo$^{1}$, G. Araujo-Pardo$^{2}$, C. Balbuena$^{3}$,
M. Bendala$^{1}$
\thanks{ \footnotesize{\em Email addresses:} ~ eabajo@us.es (E. Abajo),\, \, ~ garaujo@matem.unam.mx (G. Araujo), \, \ ~ m.camino.balbuena@upc.edu (C. Balbuena), \, \, ~ mbendala@us.es (M. Bendala)}
 \\[2ex]
\\ $^1${\footnotesize Departamento de Matem\'{a}ticas, Universidad de Sevilla, Spain.}\\
$^2${\footnotesize Instituto de Matem\'{a}ticas, Universidad Nacional Aut\'{o}noma de M\'{e}xico,} \\
{\footnotesize M\'{e}xico D. F., M\'exico }\\
$^3${\footnotesize Departament de Matem\`atica Aplicada III, Universitat
Polit\`ecnica de Catalunya, }\\
{\footnotesize Campus Nord, Edifici C2, C/ Jordi Girona 1 i 3 E-08034 Barcelona,
Spain.}
}

\begin{document}
\maketitle
\begin{abstract}
In this paper we are interested in the {\it{Cage Problem}} that consists in constructing regular graphs of given girth $g$ and minimum order. We focus on girth $g=5$, where cages are known only for degrees  $k \le 7$. We construct regular graphs of girth $5$ using techniques exposed by Funk [Note di Matematica. 29 suppl.1, (2009) 91 - 114] and Abreu et al. [Discrete Math. 312 (2012), 2832 - 2842] to obtain the best upper bounds known hitherto.  The tables given in the introduction show the improvements obtained with our results.

\end{abstract}

\begin{keyword}
Small regular graphs, cage,   girth, amalgam.\\
\emph{AMS subject classification:}  05C35,  05C38.
\end{keyword}


\section {Introduction}

All the graphs considered are finite and simple. Let $G$ be a graph with vertex set $ V =V(G)$ and edge set $E=E(G)$. The {\it{girth}} of a graph $G$ is the size $g = g(G)$ of a shortest cycle.
The degree of a vertex $v \in  V$ is the number of vertices adjacent to $v$. A graph is called $k$-regular if all its vertices have the same degree $k$, and bi-regular or $(k_1,k_2)$-regular if all its vertices have either degree $k_1$ or $k_2$. A $(k, g)$-graph is a $k$-regular graph with girth $g$ and a $(k, g)$-cage is a $(k, g)$-graph with the fewest possible number of vertices; the order of a $(k,g)$-cage is denoted by $n(k,g)$. Cages were introduced by Tutte \cite{T47} in $1947$ and their existence was proved by Erd\H{o}s and Sachs \cite{ES63} in $1963$ for any values of regularity and girth. The lower bound on the number of vertices of a $(k,g)$-graph is denoted by $n_0(k,g)$, and it is calculated using the distance partition with respect either a vertex (for odd $g$), or  and edge (for even $g$):
$$n_0(k, g) =
\left\{ \begin{array}{ll}
1 + k + k(k - 1) + \cdots + k(k - 1)^{(g-3)/2} & \makebox{if} \, g \, \makebox{is odd}; \\
2(1 + (k - 1) + \cdots + (k - 1)^{g/2 - 1})    & \makebox{if} \, g \, \makebox{is even.}
\end{array}\right.
$$

\noindent Obviously a graph that attains this lower bound is  a $(k,g)$-cage. Biggs \cite{Bi96} calls {\em excess} of a $(k,g)$--graph $G$ the difference $|V (G)| - n_0(k,g)$. There has been intense work related with {\it{The Cage Problem}}, focussed on constructing the smallest $(k,g)$-graphs (for a complete survey of this topic see \cite{ EJ08}).

\noindent In this paper we are interested in the cage problem for $g=5$, in this case $n_0(k,5)=1+k^2$. It is well known that this bound is attained for $k=2, 3, 7$ and  perhaps  for $k=57$ (see \cite{Bi96}) and that for $k=4,5,6$, the known graphs of minimum order are cages (see \cite{M99,KW79,R64,R69,YZ89,W73,W78,W82}).

\noindent  J{\o}rgensen~\cite{J05} establishes that $n( k, 5) \le 2 (q-1) ( k-2) $ for every odd prime power $q \ge 13$ and $k \le q+3$.  Abreu et al.   \cite{AABL12} prove that $n( k, 5) \le 2( q k -3 q -1)  $ for any prime  $q  \ge 13$ and $k\leq q+3$, improving  Jorgensen's bound except for $k=q+3$ where both coincide.

\noindent In \cite{F09} Funk uses a technique that consists in finding regular graphs of girth greater or equal than five and performing some operations of amalgams and reductions in the (bipartite) incidence graph, also called Levi Graph of elliptic semiplanes of type $C$ and $L$ (see \cite{AFLN3,Demb,F09}).
In this paper we improve some results of Funk finding the best possible regular graphs to amalgamate which allows us to obtain new better upper bounds. To do that, we also use the techniques given in \cite{AABL12,AABLL13}  where the authors not only amalgamate regular graphs, but also bi-regular graphs. In this paper new $(k,5)$-graphs are constructed for $17\le k\le 22$ using the incidence graphs of elliptic semiplanes of type $C$.   The new upper bounds appear  in the last column of Table \ref{table2}, which also shows the current values for $8 \le k \le 22$.
To evaluate our achievements,  we follow  the notation in \cite{EJ08,F09}, and let  $rec(k,5)$ denote  the smallest  currently  known order of a $k$-regular graph of girth $5$. Hence $n(k,5)\le rec(k,5)$.

{

\begin{table}[h!]
\scriptsize
\begin{center}

\begin{tabular}{
 |c||    @{$ \ $}  c @{$\ $}   |   @{\ } c @{\ }   |    @{\ } c @{\ }   |   @{\ } c @{\ }  |   }
\hline
 $ k  $ &$rec(k,5)$  &  Due to  & Reference   & New   $rec(k,5)$              \\
       \hline \hline

      8   &  80 & Royle, J{\o}rgensen &  \cite{Royle, J05} &      \\
  \hline
      9   &  96 &  J{\o}rgensen  &  \cite{J05} &    \\
  \hline
    10   &  124 & Exoo  &  \cite{Exoo} &    \\
  \hline
    11   &  154 &  Exoo  &  \cite{Exoo} &   \\
  \hline
    12   & 203 &  Exoo  &  \cite{Exoo} &  \\
  \hline
    13   & 230 &  Exoo  &  \cite{Exoo} &    \\
  \hline
    14   & 284 &  Abreu et al.  &  \cite{AABL12} &     \\
  \hline
    15   & 310 &  Abreu et al. &  \cite{AABL12} &     \\
  \hline
    16   & 336 &  J{\o}rgensen&  \cite{J05} &     \\
  \hline
      17   &  448 & Schwenk&  \cite{S08} & \emph{436}    \\
       \hline
     18  & 480 & Schwenk &  \cite{S08} & \emph{468}   \\
  \hline
     19  & 512 & Schwenk & \cite{S08}  & \emph{500}   \\
  \hline
     20 & 572 &   Abreu et al.  & \cite{AABL12} & \emph{564}    \\
  \hline
     21  & 682 & Abreu et al.  &  \cite{AABL12}    & \emph{666}   \\
  \hline
     22  & 720 &  J{\o}rgensen  &  \cite{J05}   & \emph{704}   \\
       \hline
\end{tabular}
 \caption{\small   \small Current and our new values of $rec(k,5)$  for $8 \le k \le 22$.}
  \label{table2}
\end{center}
\end{table}
}

\noindent The bounds obtained by Funk using elliptic semiplanes of type $L$ on $n(k,5)$ for $23 \le k \le 31 $ remain untouched whereas, for $32\le k\le 52$, we obtain the best possible regular graphs to amalgamate in this type of incidence graphs obtaining, in consequence, the best upper bounds known hitherto (see Table~\ref{table3}).

{
\begin{table}[h!]
\scriptsize
\begin{center}
\begin{tabular}  {
  |c||    @{$\,$}  c @{$\,$}   |   @{\,} c @{\,}   |    @{\,} c @{\,}   |   @{\,} c @{\,}  |  }
\hline
 $ k  $   &         $rec(k,5)$        &        Due to  &          Reference            &  New   $rec(k,5)$  \\
       \hline \hline
   32   & 1680 &   J{\o}rgensen     &      \cite{ J05} &          \emph{1624}    \\
       \hline
     33  & 1856  &  Funk   &  \cite{F09}&       \emph{ 1680}      \\
 \hline
     34 & 1920 & J{\o}rgensen & \cite{J05}             & \emph{1800}    \\
  \hline
     35 & 1984 &   Funk  & \cite{F09}            &  \emph{1860}    \\
  \hline
     36  & 2048 & Funk &   \cite{F09}            &  \emph{1920}      \\
       \hline
37  &  2514  &   Abreu et al     &  \cite{AABL12}   &   \emph{2048}     \\
  \hline
      38   &   2588    &       Abreu et al     &     \cite{AABL12}               &     \emph{2448}         \\
  \hline
      39   &    2662   &      Abreu et al     &     \cite{AABL12}               & \emph{2520} \\
  \hline
   40      &  2736      &   J{\o}rgensen &      \cite{ J05}                     & \emph{2592}   \\
  \hline
    41   &  3114 &  Abreu et al     &     \cite{AABL12}              & \emph{2664}   \\
  \hline
    42   & 3196 &  Abreu et al     &     \cite{AABL12}          & \emph{2736}   \\
  \hline
   43   & 3278   &  Abreu et al     &     \cite{AABL12}             &\emph{3040}\\
  \hline
    44   & 3360 &   J{\o}rgensen & \cite{ J05}         & \emph{3120}   \\
  \hline
    45  &  3610  & Abreu et al     &     \cite{AABL12}       & \emph{3200}  \\
  \hline
    46   & 3696 &   J{\o}rgensen & \cite{ J05}           &   \emph{3280}      \\
  \hline
   47   & 4134 &   Abreu et al     &     \cite{AABL12}           &      \emph{3360}     \\
   \hline
   48  & 4228  &   Abreu et al     &     \cite{AABL12}            &    \emph{3696}      \\
  \hline
     49 & 4322 &  Abreu et al     &     \cite{AABL12}         & \emph{4140}     \\
  \hline
     50 & 4416 &   J{\o}rgensen & \cite{J05}              &  \emph{4232}    \\
  \hline
     51  & 4704 &  J{\o}rgensen &  \cite{J05}         &   \emph{4324}     \\
 \hline
    52  &  4800  &    J{\o}rgensen & \cite{J05}                   &   \emph{4416 }   \\
   \hline
\end{tabular}
 \caption{\small   \small   Current and our new values of $rec(k,5)$  for $32 \le k \le 52$.}
  \label{table3}
\end{center}
\end{table}

\noindent Finally, when $q \ge  49$ is a prime power, the search for  $6$-regular  suitable pairs of graphs  has allowed us to  establish  the two following general results. Note that the bounds are different depending on the parity of $q$.

\begin{theorem}\label{th11}
 Given an integer $k \ge 53,$ let $q$ be the lowest odd prime power, such that $k \leq q+6$. Then  $n(k,5) \leq 2(q-1)(k-5)$.
\end{theorem}

\begin{theorem}\label{th12}
Given an integer $k \ge 68$, let $q=2^m$ be the lowest even prime, such that $k \leq q+6$. Then  $n(k,5) \leq 2 q (k-6)$.
\end{theorem}

\noindent Since the bounds of Theorem~\ref{th11} and Theorem~\ref{th12}, associated to primes $q=49$ and $q=64$, represent a considerable improvement to the current known ones, we give them explicitly.
{
 \begin{table}[h]
\scriptsize
\begin{center}
\begin{tabular}  {
  |c||    @{$\,$}  c @{$\,$}   |   @{\,} c @{\,}   |    @{\,} c @{\,}   |   @{\,} c @{\,}  |  }
\hline
 $ k  $   &         $rec(k,5)$        &        Due to  &          Reference              &  New  $rec(k,5)$  \\
       \hline \hline
   55   & 5510 &   Abreu et al    &      \cite{AABL12}          &       \emph{4800}   \\
       \hline
     70  & 8976  &  J{\o}rgensen & \cite{J05}           &   \emph{8192}     \\

   \hline
\end{tabular}
\end{center}
\end{table}

\noindent To finalize the introduction we want to empathize that Funk, in \cite{F09}, gives a pair of $4$-regular graphs of girth $5$ suitable for amalgamation in some specific incidence graphs of elliptic semiplanes and he posed the question about the existence of a pair of $5$-regular graphs with the same objective (obviously these graphs should have the same order that those given by Funk and also girth $5$), in this paper we exhibit the graphs which solve this problem.
Furthermore, let us notice that all the bounds on $n(k,5)$ contained in this paper are obtained constructively, that is, for each $k$, we construct a $(k,5)$-graph with order $new \ rec(k,5)$.



\section{Preliminaries}
A useful tool to construct $k$-regular graphs of girth 5  is the operation of \emph{amalgamation} on the incidence graph of an elliptic semiplane (J{\o}rgensen \cite{J05} 2005 and Funk \cite{F09} 2009).

\noindent For $q=p^m$ a prime power, consider the Levi graphs $C_q$ and $L_q$ of the so-called elliptic semiplanes of type \emph{C} and \emph{L}, respectively. Recall that the semiplane of type $\emph{C}$ is obtained from the projective plane over the field $\mathbb{F}_q$ by deleting a point and all the lines incident with it together with all the points that belong to one of these lines, and that the Levi graph $C_q$ is bipartite, $q$-regular and has $2q^2$ vertices which corresponds in the elliptic semiplane to $q^2$ points and $q^2$ lines both partitioned into $q$ parallel classes or blocks of $q$ elements each. On the other hand, the semiplane of type  $\emph{L}$ is obtained by deleting from the projective plane a point and all the lines incident with it together with a different line and all its points, here the Levi graph of $L_q$ is also bipartite, $q$-regular and has $2(q^2-1)$ vertices of which $q^2-1$ are points and $q^2-1$ are lines in the elliptic semiplane, both partitioned into $q+1$ parallel classes of $q-1$ elements each.

\noindent The construction of our new graphs consists in finding regular and bi-regular graphs of girth greater or equal than five and performing some operations of amalgams and reductions in $C_q$ or $L_q$. 
In \cite{J05}, J{\o}rgensen exploits these ideas and proves that two graphs are suitable for amalgamation (one of them in each block of points and the other in each block of lines) if they have disjoint sets of Cayley colors.

\noindent In \cite{AABL12}   these ideas are also used to construct graphs using the elliptic semiplane of type {\emph{C}}, and the main Theorem of \cite{AABL12} was refined in \cite{AABLL13} to construct bi-regular cages of girth 5. In fact,  the suitable graphs to amalgamate can have some common Cayley color, but only for some specific edges.

\noindent The paper is organized as follows. In the next Section we work with elliptic semiplanes of type~{\emph{C}} and with techniques used in \cite{AABL12, AABLL13}, in Section 4 with elliptic semiplanes of type {\emph{L} and with techniques given by Funk in  \cite{F09}. Finally, in Section 5 we return to the elliptic semiplanes of type {\emph{C}} for even primes because they require new descriptions.

\section {Amalgamating  into elliptic semiplanes of type \emph{C}}

Let $q$ be a prime power and $\mathbb{F}_q$ the finite field of order $q$; we recall the definition and properties of the incidence bipartite graph $C_q$  of an elliptic semiplane of type \emph{C} exactly as they appear in \cite{AABL12, AABLL13}. Notice that in these papers the authors call this graph $B_q$ and here, as it is related to the elliptic semiplane of type \emph{C}, we prefer to call it $C_q$.

 \begin{definition}\label{defCq}
 Let $C_q$ be a bipartite graph with vertex set
$(V_0,V_1)$ where $V_r=\mathbb{F}_q\times \mathbb{F}_q$, $r=0,1$; and the
edge set defined as follows:
\begin{equation}\label{Bq}
(x,y)_0\in V_0 \mbox{ adjacent to } (m,b)_1 \in V_1  \mbox{ if and only if } y=mx+b.
\end{equation}
\end{definition}

\noindent The graph $C_q$ is also known  as the incidence graph of the biaffine
plane \cite{H04} and it has been used in the problem of finding extremal graphs without short cycles (see \cite{AABL12, AABLL13,  AFLN1, AFLN2, AFLN3, AB11, ABH10,AGMS07, B08, BMSZ13,LU95}).
The graph $C_q$ is $q$-regular of order $2q^2$, has girth   $g =  6$  for $q\ge 3$ and it is  vertex transitive. Other properties of the  graph $C_q$  are well known (see
\cite{AABL12, AABLL13, H04,LU95}).

\noindent  Let $\Gamma_1$ and $\Gamma_2$ be two graphs of the same order and with the same labels on their vertices; an \emph{amalgam of $\Gamma_1$ into $\Gamma_2$} is a graph obtained adding all the edges of $\Gamma_1$ into $\Gamma_2$.
 In \cite{AABL12} it is described a technique of amalgamation of two $r$-regular graphs $H_0, H_1$ and two $(r,r+1)$-regular graphs $G_0, G_1$ (all of them of girth at least $5$ and with some specific properties) into a subgraph of $C_q$  obtaining the resulting amalgam graph, denoted by
$\mathcal{C}_q(H_0,H_1,G_0,G_1)$, which is $(q+r)$-regular and has girth at least five.

\noindent Theorem \ref{teoC} is a reformulation of Theorem~$5$ in \cite{AABL12} (with a new strong hypothesis that also appears in Theorem~$4.9$ in \cite{AABLL13}).  Recall that if $G$ is a graph with $V(G)$ labeled with the elements of $\mathbb{F}_q$ and $\alpha\beta$ is an edge of $G$, then the \emph{Cayley Color} or \emph{weight} of $\alpha\beta$ is $\pm (\alpha-\beta)\in \mathbb{F}_q-\{0\}$.

\begin{theorem} \label{teoC}

Let $q\geq 3$ be a prime power and $r \ge 2 $  an integer. Consider graphs $H_0$,  $H_1$,  $G_0$ and $G_1$ with the following properties:
\begin{itemize}
\item[$(i) $]    $V(G_i) = \mathbb{F}_q$ and $G_i$ is an $( r,r+1)$-regular graph of girth $ g(G_i) \ge 5$ for $i=0,1$;

\item[$(ii)$] $H_i$ is an $r$-regular graph of girth $g(H_i) \ge 5$ and $V(H_i)= \{v \in \mathbb{F}_q  : d_{G_j}(v)=r \ \mbox{with}\  i\not= j \}$, for $i,j\in \{0,1\}$.

\item[ $(iii)$]   $E(H_0) \cap  E(H_1) = \emptyset$, $E(H_0) \cap  E(G_1) = \emptyset$,
$E(H_1) \cap  E(G_0) = \emptyset$  and   $G_0$ and $G_1$ have disjoint  Cayley colors.

\end{itemize}

\noindent  Consider  the sets of vertices $ P'_0=  \{  ( 0 ,y )_0 : y  \in V(H_0) \}$,  $L'_0=  \{  ( 0,b )_1  : b \in V(H_1) \}$; for all $x,m\in  \mathbb{F}_q-0$, let  $P_x=  \{  ( x ,y )_0:  y \in \mathbb{F}_q  \} $
and
$L_m=  \{  ( m,b )_1  : b \in \mathbb{F}_q  \}$.  Let $ A= P'_0 \cup ( \bigcup_{x \in \mathbb{F}_q-0}   P_x) \cup L'_0 \cup ( \bigcup_{m \in \mathbb{F}_q-  0  }L_m)  $     and consider the induced subgraph $\mathcal{C}_q[A]$ where $\mathcal{C}_q$ is the graph given in Definition \ref{Bq}.

Moreover, let the sets
of edges $ E_0(0)=  \{  ( 0 ,y )_0  ( 0, y' )_0 : yy ' \in E(H_0) \},$ $E_1(0)=  \{  ( 0,b )_1  ( 0,b' )_1 : bb'\in E(H_1) \},$
$ E_0(x)=  \{  ( x ,y )_0  ( x, y' )_0 : yy ' \in E(G_0) \},$
$E_1(m)=  \{  ( m,b )_1  ( m,b' )_1 : bb' \in E(G_1) \}$ for all $m, x  \in  \mathbb{F}_q - 0 .$

The graph $\mathcal{C}_q(H_0,H_1,G_0,G_1)$ with vertex set $A$ and edge set $E(C_q[A]) \cup ( \bigcup_{x \in \mathbb{F}_q}   E_0(x) )  \cup ( \bigcup_{m \in \mathbb{F}_q}
  E_1(m) )  $ is $(q+r)$-regular and has girth at least five.

\end{theorem}
The proof is the same as the one of Theorem~$4.9$ in \cite{AABLL13}. Notice that Theorem \ref{teoC} can also be applied when $G_0$ and $G_1$ are regular graphs (then $H_0=G_0$ and $H_1=G_1$). In this case we denote the resulting graph by $\mathcal{C}_q(G_0,G_1).$

Next, for primes $q \in \{16, 17, 19 \}$, we construct graphs $H_0$, $H_1$, $G_0$, $G_1$, which verify  the conditions of Theorem  \ref{teoC}.

\noindent{\bf{Construction 1:}}
\begin{itemize}
\item For $q=16$:

Let $(\mathbb{F}_{16},+) \cong  ( (\mathbb{Z}_2)^4,+)$ be a finite field of order $16$ with set of elements $\{  (d,e,f,g ) : d,e,f,g \in \mathbb{Z}_2\}$, we write $defg$ instead of $ (d,e,f,g )$.
Consider the graphs $H_0$, $H_1$, $G_0$ and $G_1$ displayed in Figure~\ref{fig:caseq16}.
The graphs $G_0$ and $G_1$ are not isomorphic, although both have girth 5 and order 16, with  6 vertices of degree $4$ and 10 vertices of degree $3$.
The labeling of the vertices of $G_0$ and $G_1$ is such that the vertices of the set $S= \{ 0 0 0 0, 1 1 0 0, 0 1 1 0, 1 0 0 1, 0 0 1 1,
1 1 1 1\}$  have degree four and  the other ones have degree three. The weights or Cayley colors of $G_0$ (and $G_1$) are $\{0001, 0 0 1 0, 0 1 0 0, 1 0 0 0, 1 1 1 1\}$ (and $ \{ 0 0 1 1, 0 1 1 0, 0 1 1 1, 1 0 0 1, 1 0 1 0,
1 0 1 1, 1 1 0 0, 1 1 0 1, 1 1 1 0  \}$, respectively).   Hence, $G_0$ and $G_1$ have disjoint sets of Cayley colors.
Moreover, the graphs $H_0$ and $H_1$ are isomorphic to the Petersen graph and they are labeled with the elements of $(\mathbb{Z}_2)^4- S$ in such a way that   $E(H_0) \cap  E(H_1) = \emptyset$, $E(H_0) \cap  E(G_1) = \emptyset$ and
$E(H_1) \cap  E(G_0) = \emptyset$.



\begin{figure}[h!]
\begin{center}
\scalebox{1.05}
{\begin{pspicture}(1.9, -7.7 )(3., 2.60)

\definecolor{color0}{rgb}{0.0,0.,0.}

\rput(  0.,  2.4  ){\tiny $H_0$}
\rput( 0,   -2.5   ){\tiny $G_0$}
\rput(  5.5,  2.4  ){\tiny $H_1$}
\rput( 5.5,   -2.5   ){\tiny $G_1$}

\pscircle[linewidth=0.013,dimen=outer, fillstyle=solid,fillcolor=color0](0., 2.   ){0.1}
\pscircle[linewidth=0.013,dimen=outer, fillstyle=solid,fillcolor=color0](-1.9, 0.62)  {0.1}
\pscircle[linewidth=0.013,dimen=outer, fillstyle=solid,fillcolor=color0](-1.2, -1.6)  {0.1}
\pscircle[linewidth=0.013,dimen=outer, fillstyle=solid,fillcolor=color0]( 1.2, -1.6 )  {0.1}
\pscircle[linewidth=0.013,dimen=outer, fillstyle=solid,fillcolor=color0]( 1.9, 0.62 )   {0.1}

\pscircle[linewidth=0.013,dimen=outer, fillstyle=solid,fillcolor=color0]( 0., 1.1 ){0.1}
\pscircle[linewidth=0.013,dimen=outer, fillstyle=solid,fillcolor=color0]( -1.04, 0.34 ){0.1}
\pscircle[linewidth=0.013,dimen=outer, fillstyle=solid,fillcolor=color0]( -0.65, -0.89 ){0.1}
\pscircle[linewidth=0.013,dimen=outer, fillstyle=solid,fillcolor=color0]( 0.65, -0.89 ){0.1}
\pscircle[linewidth=0.013,dimen=outer, fillstyle=solid,fillcolor=color0 ]( 1.05, 0.34 ){0.1}


\psline[linewidth=0.013cm](0., 2.  )  (  -1.9, 0.62  )
\psline[linewidth=0.013cm]( -1.9, 0.62 )  (-1.2, -1.6)
\psline[linewidth=0.013cm]( -1.2, -1.6 )  (1.2, -1.6  )
\psline[linewidth=0.013cm](1.2, -1.6 )  ( 1.9, 0.62  )
\psline[linewidth=0.013cm]( 1.9, 0.62 )  (0., 2.   )


\psline[linewidth=0.013cm](0., 1.1 )  (-0.65, -0.89 )
\psline[linewidth=0.013cm]( -0.65, -0.89 )  ( 1.05, 0.34 )
\psline[linewidth=0.013cm]( 1.05, 0.34 )  ( -1.04, 0.34  )
\psline[linewidth=0.013cm](  -1.04, 0.34 )  (0.65, -0.89  )
\psline[linewidth=0.013cm]( 0.65, -0.89 )  (0., 1.1)


\psline[linewidth=0.013cm]( 0., 2. )( 0., 1.1 )
\psline[linewidth=0.013cm]( -1.9, 0.62 ) ( -1.04, 0.34 )
\psline[linewidth=0.013cm]( -1.2, -1.6  )( -0.65, -0.89 )
\psline[linewidth=0.013cm](1.2, -1.6 )( 0.65, -0.89 )
\psline[linewidth=0.013cm]( 1.9, 0.62 )(1.05, 0.34 )

\rput( -0.4, 2.0 ){\tiny $0100$}
\rput(-2, 0.79 ){\tiny $0111$}
\rput( -1.2, -1.79 ){\tiny $1110$}
\rput(  1.2, -1.79  ){\tiny $1010$}
\rput( 2.05, 0.79  ){\tiny $1011$}

\rput(  -0.4, 1.2  ){\tiny $0101$}
\rput(  -1.0, 0.55  ){\tiny $0010$}
\rput(  -0.9, -0.7  ){\tiny $1101$}
\rput(  0.9, -0.7  ){\tiny $0001$}
\rput( 1.0, 0.55  ){\tiny $1000$}


\pscircle[linewidth=0.013,dimen=outer, fillstyle=solid,fillcolor=color0](5.5 , 2.   ){0.1}
\pscircle[linewidth=0.013,dimen=outer, fillstyle=solid,fillcolor=color0]( 3.6  , 0.62)  {0.1}
\pscircle[linewidth=0.013,dimen=outer, fillstyle=solid,fillcolor=color0](4.3, -1.6)  {0.1}
\pscircle[linewidth=0.013,dimen=outer, fillstyle=solid,fillcolor=color0]( 6.7, -1.6 )  {0.1}
\pscircle[linewidth=0.013,dimen=outer, fillstyle=solid,fillcolor=color0]( 7.4, 0.62 )   {0.1}

\pscircle[linewidth=0.013,dimen=outer, fillstyle=solid,fillcolor=color0]( 5.5, 1.1 ){0.1}
\pscircle[linewidth=0.013,dimen=outer, fillstyle=solid,fillcolor=color0]( 4.46, 0.34 ){0.1}
\pscircle[linewidth=0.013,dimen=outer, fillstyle=solid,fillcolor=color0]( 4.85, -0.89 ){0.1}
\pscircle[linewidth=0.013,dimen=outer, fillstyle=solid,fillcolor=color0](  6.15 , -0.89 ){0.1}
\pscircle[linewidth=0.013,dimen=outer, fillstyle=solid,fillcolor=color0 ]( 6.55 , 0.34 ){0.1}


\psline[linewidth=0.013cm](5.5 , 2.    )  (  3.6  , 0.62  )
\psline[linewidth=0.013cm]( 3.6  , 0.62 )  (4.3, -1.6)
\psline[linewidth=0.013cm]( 4.3, -1.6 )  (6.7, -1.6  )
\psline[linewidth=0.013cm](6.7, -1.6 )  ( 7.4, 0.62  )
\psline[linewidth=0.013cm]( 7.4, 0.62 )  ( 5.5,2. )


\psline[linewidth=0.013cm]( 5.5, 1.1  )  (4.85, -0.89 )
\psline[linewidth=0.013cm]( 4.85, -0.89 )  ( 6.55 , 0.34 )
\psline[linewidth=0.013cm]( 6.55 , 0.34 )  ( 4.46, 0.34  )
\psline[linewidth=0.013cm]( 4.46, 0.34 )  (6.15 , -0.89 )
\psline[linewidth=0.013cm]( 6.15 , -0.89 )  ( 5.5, 1.1 )


\psline[linewidth=0.013cm]( 5.5 , 2.  )( 5.5, 1.1 )
\psline[linewidth=0.013cm](3.6  , 0.62 ) ( 4.46, 0.34 )
\psline[linewidth=0.013cm]( 4.3, -1.6  )( 4.85, -0.89 )
\psline[linewidth=0.013cm](6.7, -1.6 )(6.15 , -0.89 )
\psline[linewidth=0.013cm]( 7.4, 0.62 )( 6.55 , 0.34 )

\rput( 5.1, 2.0 ){\tiny $1000$}
\rput(  3.5 , 0.79 ){\tiny $1010$}
\rput( 4.3, -1.79 ){\tiny $0101$}
\rput(  6.7 , -1.79  ){\tiny $0111$}
\rput( 7.55, 0.79  ){\tiny $0001$}

\rput(   5.1   , 1.2  ){\tiny $1110$}
\rput(  4.5 , 0.55  ){\tiny $1101$}
\rput(  4.6, -0.7  ){\tiny $0010$}
\rput(  6.4, -0.7  ){\tiny $1011$}
\rput( 6.5, 0.55  ){\tiny $0100$}


\pscircle[linewidth=0.013,dimen=outer, fillstyle=solid,fillcolor=color0](  0, -3.0  ){0.1}                     
\pscircle[linewidth=0.013,dimen=outer, fillstyle=solid,fillcolor=color0](  -0.765, -3.15 ){0.1}            
\pscircle[linewidth=0.013,dimen=outer, fillstyle=solid,fillcolor=color0](  -1.41, -3.59  ){0.1}
\pscircle[linewidth=0.013,dimen=outer, fillstyle=solid,fillcolor=color0](  -1.85, -4.23   ){0.1}
\pscircle[linewidth=0.013,dimen=outer, fillstyle=solid,fillcolor=color0]( -2.00, -5.00 ){0.1}

\pscircle[linewidth=0.013,dimen=outer, fillstyle=solid,fillcolor=color0]( -1.85, -5.77  ){0.1}                
\pscircle[linewidth=0.013,dimen=outer, fillstyle=solid,fillcolor=color0](  -1.41, -6.41  ){0.1}
\pscircle[linewidth=0.013,dimen=outer, fillstyle=solid,fillcolor=color0](  -0.765, -6.85   ){0.1}
\pscircle[linewidth=0.013,dimen=outer, fillstyle=solid,fillcolor=color0]( 0, -7.00  ){0.1}
\pscircle[linewidth=0.013,dimen=outer, fillstyle=solid,fillcolor=color0](  0.765, -6.85 ){0.1}

\pscircle[linewidth=0.013,dimen=outer, fillstyle=solid,fillcolor=color0](1.41, -6.41  ){0.1}       
\pscircle[linewidth=0.013,dimen=outer, fillstyle=solid,fillcolor=color0]( 1.85, -5.77  ){0.1}
\pscircle[linewidth=0.013,dimen=outer, fillstyle=solid,fillcolor=color0](  2.00, -5.00   ){0.1}
\pscircle[linewidth=0.013,dimen=outer, fillstyle=solid,fillcolor=color0]( 1.85, -4.23  ){0.1}
\pscircle[linewidth=0.013,dimen=outer, fillstyle=solid,fillcolor=color0]( 1.41, -3.59   ){0.1}

\pscircle[linewidth=0.013,dimen=outer, fillstyle=solid,fillcolor=color0]( 0.765, -3.15 ){0.1}       

\rput(  0, -2.8 ){\tiny $1010$}

\rput(  -0.7, -2.94 ){\tiny $1110$}
\rput(-1.48, -3.4  ){\tiny $1100$}
\rput( -1.99, -3.97   ){\tiny $1101$}
\rput(  -1.67, -4.83 ){\tiny $1111$}
\rput(  -2.1, -5.97 ){\tiny $0111$}
\rput(  -1.65, -6.60 ){\tiny $0110$}

\rput( -0.84, -7.1  ){\tiny $0100$}
\rput(  0.0, -7.24   ){\tiny $0101$}
\rput(  0.84, -7.1   ){\tiny $0001$}

\rput(  1.65, -6.60  ){\tiny $0011$}
\rput( 2.1, -5.97 ){\tiny $0010$}
\rput(  1.67, -4.83 ){\tiny $0000$}
\rput(  1.99, -3.97  ) {\tiny $1000$}
\rput(  1.59, -3.4   ){\tiny $1001$}
\rput(  0.77, -2.94  ){\tiny $1011$}


\psline[linewidth=0.013cm](0, -3.0  )  (-0.765, -3.15  )
\psline[linewidth=0.013cm]( -0.765, -3.15 )  ( -1.41, -3.59 )
\psline[linewidth=0.013cm]( -1.41, -3.59 )  ( -1.85, -4.23 )
\psline[linewidth=0.013cm]( -1.85, -4.23 )  ( -2.00, -5.00  )
\psline[linewidth=0.013cm](  -2.00, -5.00 )  ( -1.85, -5.77  )

\psline[linewidth=0.013cm]( -1.85, -5.77 )  (  -1.41, -6.41 )
\psline[linewidth=0.013cm](  -1.41, -6.41 )  (-0.765, -6.85 )
\psline[linewidth=0.013cm]( -0.765, -6.85 )  ( 0, -7.00 )
\psline[linewidth=0.013cm]( 0, -7.00 )  (  0.765, -6.85  )
\psline[linewidth=0.013cm](  0.765, -6.85 )  ( 1.41, -6.41 )

\psline[linewidth=0.013cm](1.41, -6.41 )  ( 1.85, -5.77 )
\psline[linewidth=0.013cm]( 1.85, -5.77 )  (  2.00, -5.00  )
\psline[linewidth=0.013cm](  2.00, -5.00  )  ( 1.85, -4.23 )
\psline[linewidth=0.013cm]( 1.85, -4.23)  ( 1.41, -3.59  )
\psline[linewidth=0.013cm](1.41, -3.59 )  (  1.41, -3.59 )

\psline[linewidth=0.013cm](  1.41, -3.59 )  (0.765, -3.15  )
\psline[linewidth=0.013cm]( 0.765, -3.15 )  ( 0, -3.0  )


\psline[linewidth=0.013cm]( 0, -3.0 )  ( 1.85, -5.77 )
\psline[linewidth=0.013cm]( -0.765, -3.15 )  ( -1.41, -6.41 )
\psline[linewidth=0.013cm]( -1.41, -3.59 )  (1.85, -4.23 )
\psline[linewidth=0.013cm]( -1.41, -3.59 )  (1.41, -6.41 )

\psline[linewidth=0.013cm]( -1.85, -4.23 )  ( 0, -7.00 )

\psline[linewidth=0.013cm]( -2.00, -5.00 )  ( 0.765, -3.15 )
\psline[linewidth=0.013cm](  -2.00, -5.00 )  ( 2.00, -5.00 )
\psline[linewidth=0.013cm](-1.85, -5.77 )  ( 1.41, -6.41 )

\psline[linewidth=0.013cm]( -1.41, -6.41 )  (1.41, -3.59 )
\psline[linewidth=0.013cm]( -0.765, -6.85)  (2.00, -5.00 )
\psline[linewidth=0.013cm]( 0.765, -6.85 )  (1.41, -3.59 )


\pscircle[linewidth=0.013,dimen=outer, fillstyle=solid,fillcolor=color0](  5.5, -3.0  ){0.1}                     
\pscircle[linewidth=0.013,dimen=outer, fillstyle=solid,fillcolor=color0](  4.735, -3.15 ){0.1}            
\pscircle[linewidth=0.013,dimen=outer, fillstyle=solid,fillcolor=color0](   4.09 , -3.59  ){0.1}
\pscircle[linewidth=0.013,dimen=outer, fillstyle=solid,fillcolor=color0](  3.65 , -4.23   ){0.1}
\pscircle[linewidth=0.013,dimen=outer, fillstyle=solid,fillcolor=color0]( 3.5, -5.00 ){0.1}

\pscircle[linewidth=0.013,dimen=outer, fillstyle=solid,fillcolor=color0](   3.65  , -5.77  ){0.1}                
\pscircle[linewidth=0.013,dimen=outer, fillstyle=solid,fillcolor=color0](  4.09  , -6.41  ){0.1}
\pscircle[linewidth=0.013,dimen=outer, fillstyle=solid,fillcolor=color0](   4.735 , -6.85   ){0.1}
\pscircle[linewidth=0.013,dimen=outer, fillstyle=solid,fillcolor=color0]( 5.5 , -7.00  ){0.1}
\pscircle[linewidth=0.013,dimen=outer, fillstyle=solid,fillcolor=color0](  6.265  , -6.85 ){0.1}

\pscircle[linewidth=0.013,dimen=outer, fillstyle=solid,fillcolor=color0](  6.91  , -6.41  ){0.1}       
\pscircle[linewidth=0.013,dimen=outer, fillstyle=solid,fillcolor=color0](    7.35  , -5.77  ){0.1}
\pscircle[linewidth=0.013,dimen=outer, fillstyle=solid,fillcolor=color0](   7.5  , -5.00   ){0.1}
\pscircle[linewidth=0.013,dimen=outer, fillstyle=solid,fillcolor=color0](   7.35 , -4.23  ){0.1}
\pscircle[linewidth=0.013,dimen=outer, fillstyle=solid,fillcolor=color0](  6.91 , -3.59   ){0.1}

\pscircle[linewidth=0.013,dimen=outer, fillstyle=solid,fillcolor=color0]( 6.265 , -3.15 ){0.1}       

\rput(  5.5, -2.8 ){\tiny $1101$}

\rput( 4.8, -2.94 ){\tiny $0000$}
\rput(  4.02 , -3.4  ){\tiny $0111$}
\rput(  3.51  , -3.97   ){\tiny $1100$}
\rput(   3.83 , -4.83 ){\tiny $0001$}
\rput(   3.4  , -5.97 ){\tiny $1011$}
\rput(   3.85  , -6.60 ){\tiny $0010$}

\rput(  4.66  , -7.1  ){\tiny $1001$}
\rput(  5.5, -7.24   ){\tiny $1010$}
\rput(  6.34  , -7.1   ){\tiny $0110$}

\rput(  7.15 , -6.60  ){\tiny $1000$}
\rput(   7.6 , -5.97 ){\tiny $1111$}
\rput(  7.17  , -4.83 ){\tiny $0100$}
\rput(   7.49, -3.97  ) {\tiny $1110$}
\rput(   7.09 , -3.4   ){\tiny $0101$}
\rput(  6.27 , -2.94  ){\tiny $0011$}


\psline[linewidth=0.013cm](5.5, -3.0  )  ( 4.735 , -3.15  )
\psline[linewidth=0.013cm]( 4.735 , -3.15 )  ( 4.09, -3.59 )
\psline[linewidth=0.013cm]( 4.09 , -3.59 )  (  3.65, -4.23 )
\psline[linewidth=0.013cm]( 3.65  , -4.23 )  (  3.5, -5.00  )
\psline[linewidth=0.013cm](  3.5 , -5.00 )  (  3.65 , -5.77  )

\psline[linewidth=0.013cm](    3.65  , -5.77 )  (   4.09  , -6.41 )
\psline[linewidth=0.013cm](  4.09  , -6.41 )  (  4.735  , -6.85 )
\psline[linewidth=0.013cm]( 4.735, -6.85 )  ( 5.5, -7.00 )
\psline[linewidth=0.013cm]( 5.5, -7.00 )  (  6.265, -6.85  )
\psline[linewidth=0.013cm](  6.265, -6.85 )  (  6.91  , -6.41 )

\psline[linewidth=0.013cm](   6.91  , -6.41 )  (  7.35 , -5.77 )
\psline[linewidth=0.013cm]( 7.35 , -5.77 )  ( 7.50, -5.00  )
\psline[linewidth=0.013cm](  7.50, -5.00  )  (  7.35  , -4.23 )
\psline[linewidth=0.013cm](  7.35  , -4.23)  ( 6.91 , -3.59  )
\psline[linewidth=0.013cm]( 6.91, -3.59 )  (  6.91  , -3.59 )

\psline[linewidth=0.013cm](  6.91  , -3.59 )  (  6.265, -3.15  )
\psline[linewidth=0.013cm]( 6.265, -3.15 )  ( 5.5, -3.0  )


\psline[linewidth=0.013cm]( 5.5, -3.00 )  (3.65224, -5.77 )
\psline[linewidth=0.013cm]( 6.26537, -3.15 )  ( 4.73463, -6.85 )
\psline[linewidth=0.013cm]( 6.26537, -3.15 )  (6.91421, -6.41 )
\psline[linewidth=0.013cm]( 6.91421, -3.59 )  (3.65224, -4.23 )

\psline[linewidth=0.013cm]( 4.73463, -3.15 )  ( 7.34776, -4.23 )

\psline[linewidth=0.013cm]( 4.08579, -3.59 )  ( 4.73463, -6.85 )
\psline[linewidth=0.013cm]( 3.65224, -4.23  )  (7.34776, -5.77 )
\psline[linewidth=0.013cm](3.5, -5.00  )  ( 6.26537, -6.85 )

\psline[linewidth=0.013cm]( 4.08579, -6.41  )  ( 7.34776, -5.77 )
\psline[linewidth=0.013cm](5.5, -7.00)  (7.5, -5.00 )
\psline[linewidth=0.013cm]( 4.73463, -3.15 )  (6.26537, -6.85 )

\end{pspicture}}
\vskip -0.2cm
\caption{\small   The graphs $H_i$ and $G_i$ for $i=0,1$ where $q=16$.
\label{fig:caseq16}}
\end{center}
\end{figure}
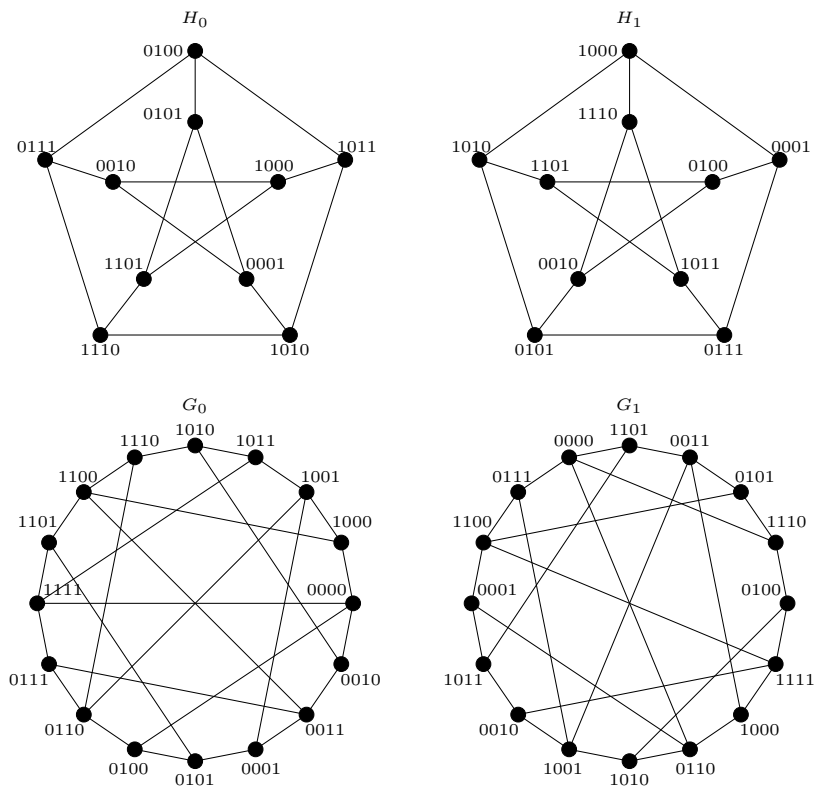

\item For $q=17$:

Let $\mathbb{Z}_{17}$ be a finite field of order $17$ and consider the graphs $H_0$, $H_1$, $G_0$ and $G_1$ displayed in Figure~\ref{fig:caseq17}.
The graphs  $G_0$ and $G_1$ are isomorphic, with  7 vertices of degree $4$ and 10 vertices of degree $3$.
Both graphs have the same set of vertices $S= \{0,2,5,8,10,13,15 \}$ of degree $4$, and the Cayley colors of $G_0$ (and $G_1$) are $ \pm \{ 1,5,8 \}$ (and $ \pm \{2,3,4,6,7\}$, respectively).
Regarding $H_0$ and $H_1$,  they are labeled with the elements of $\mathbb{Z}_{17}- S$ and verify   $E(H_0) \cap  E(H_1) = \emptyset$, $E(H_0) \cap  E(G_1) = \emptyset$ and
$E(H_1) \cap  E(G_0) = \emptyset$.




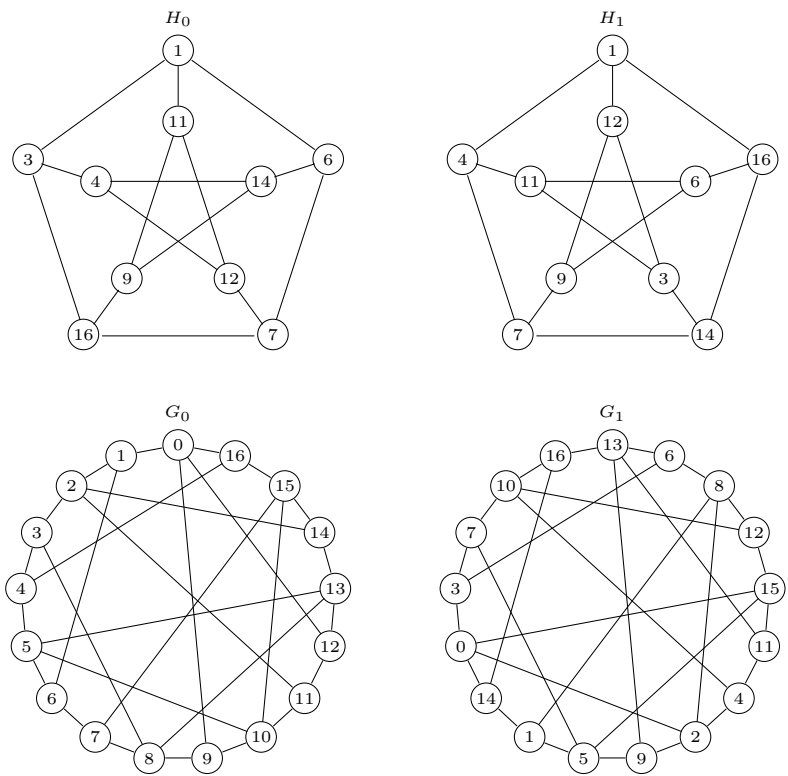
\begin{figure}[h!]
\begin{center}
\scalebox{1.05}
{\begin{pspicture}(1.9, -7.7 )(3., 2.60)

\definecolor{color0}{rgb}{0.0,0.,0.}

\rput(  0.,  2.4  ){\tiny$H_0$}
\rput( 0,   -2.6   ){\tiny $G_0$}
\rput(  5.5,  2.4  ){\tiny$H_1$}
\rput( 5.5,   -2.6   ){\tiny $G_1$}

\pscircle[linewidth=0.013,dimen=outer](0., 2.){0.2}
\pscircle[linewidth=0.013,dimen=outer](-1.9, 0.62){0.2}
\pscircle[linewidth=0.013,dimen=outer](-1.2, -1.6){0.2}
\pscircle[linewidth=0.013,dimen=outer]( 1.2, -1.6 ){0.2}
\pscircle[linewidth=0.013,dimen=outer]( 1.9, 0.62 ){0.2}

\pscircle[linewidth=0.013,dimen=outer]( 0., 1.1 ){0.2}
\pscircle[linewidth=0.013,dimen=outer]( -1.04, 0.34 ){0.2}
\pscircle[linewidth=0.013,dimen=outer]( -0.65, -0.89 ){0.2}
\pscircle[linewidth=0.013,dimen=outer]( 0.65, -0.89 ){0.2}
\pscircle[linewidth=0.013, dimen=outer]( 1.05, 0.34 ){0.2}


\psline[linewidth=0.013cm]( -0.17119, 1.87562 )  (  -1.73092, 0.742411  )
\psline[linewidth=0.013cm]( -1.83672, 0.416788 )  (-1.24096, -1.41679)
\psline[linewidth=0.013cm]( -0.963968, -1.61803 )  (0.963968, -1.61803  )
\psline[linewidth=0.013cm]( 1.24096, -1.41679 )  ( 1.83672, 0.416788  )
\psline[linewidth=0.013cm]( 0.17119, 1.87562 )  ( 1.73092, 0.742411  )


\psline[linewidth=0.013cm](-0.0581907, 0.920907 )  ( -0.588373, -0.710826 )
\psline[linewidth=0.013cm]( -0.893817, 0.229233 )  ( 0.494218, -0.779233 )
\psline[linewidth=0.013cm]( -0.857853, 0.339919 )  ( 0.857853, 0.339919  )
\psline[linewidth=0.013cm](  -0.494218, -0.779233 )  ( 0.893817, 0.229233  )
\psline[linewidth=0.013cm]( 0.0581907, 0.920907 )  (0.588373, -0.710826)


\psline[linewidth=0.013cm]( 0,  1.8 )( 0,1.3 )
\psline[linewidth=0.013cm]( -1.73092, 0.562411 ) ( -1.21735, 0.395542 )
\psline[linewidth=0.013cm]( 1.73092, 0.562411  )( 1.21735, 0.395542 )
\psline[linewidth=0.013cm]( 1.06977, -1.47241 )( 0.752365, -1.03554 )
\psline[linewidth=0.013cm]( -1.06977, -1.47241 )( -0.752365, -1.03554 )

\rput( 0, 2.0 ){\tiny $1$}
\rput(-1.9, 0.62 ){\tiny $3$}
\rput( -1.2, -1.6 ){\tiny $16$}
\rput(  1.2, -1.6  ){\tiny $7$}
\rput( 1.9, 0.62  ){\tiny $6$}

\rput(  0., 1.1  ){\tiny $11$}
\rput(  -1.04616, 0.339919  ){\tiny $4$}
\rput(  -0.646564, -0.889919  ){\tiny $9$}
\rput(  0.646564, -0.889919  ){\tiny $12$}
\rput( 1.04616, 0.339919  ){\tiny $14$}

\pscircle[linewidth=0.013,dimen=outer](  5.5, 2.  ){0.2}
\pscircle[linewidth=0.013,dimen=outer]( 3.6, 0.62){0.2}
\pscircle[linewidth=0.013,dimen=outer]( 4.3, -1.6){0.2}
\pscircle[linewidth=0.013,dimen=outer]( 6.7, -1.6 ){0.2}
\pscircle[linewidth=0.013,dimen=outer]( 7.4, 0.62 ){0.2}

\pscircle[linewidth=0.013,dimen=outer]( 5.5, 1.1 ){0.2}
\pscircle[linewidth=0.013,dimen=outer]( 4.46, 0.34 ){0.2}
\pscircle[linewidth=0.013,dimen=outer]( 4.85, -0.89 ){0.2}
\pscircle[linewidth=0.013,dimen=outer]( 6.15, -0.89 ){0.2}
\pscircle[linewidth=0.013, dimen=outer]( 6.55, 0.34 ){0.2}


\psline[linewidth=0.013cm]( 5.32881, 1.87562 )  ( 3.76908, 0.742411  )
\psline[linewidth=0.013cm]( 3.66328, 0.416788 )  (4.25904, -1.41679)
\psline[linewidth=0.013cm]( 4.53603, -1.61803 )  (6.46397, -1.61803  )
\psline[linewidth=0.013cm]( 6.74096, -1.41679 )  ( 7.33672, 0.416788  )
\psline[linewidth=0.013cm]( 5.67119, 1.87562 )  ( 7.23, 0.742411   )


\psline[linewidth=0.013cm]( 5.44181, 0.920907 )  (  4.91163 , -0.710826 )
\psline[linewidth=0.013cm](  4.6  , 0.229233 )  ( 5.99422, -0.779233 )
\psline[linewidth=0.013cm](  4.64215  , 0.339919 )  ( 6.35785 , 0.339919  )
\psline[linewidth=0.013cm](   5.00578  , -0.779233 )  ( 6.39382 , 0.229233  )
\psline[linewidth=0.013cm]( 5.5581907, 0.920907 )  (6.08837  , -0.710826)


\psline[linewidth=0.013cm]( 5.5,  1.8 )( 5.5, 1.3 )
\psline[linewidth=0.013cm]( 3.76908, 0.562411 ) (4.28265,  0.395542 )
\psline[linewidth=0.013cm]( 7.2309, 0.562411  )( 6.71735, 0.395542 )
\psline[linewidth=0.013cm]( 6.56977, -1.47241 )( 6.25237, -1.03554 )
\psline[linewidth=0.013cm]( 4.43023, -1.47241 )( 4.74764, -1.03554 )

\rput( 5.5, 2.0 ){\tiny $1$}
\rput(3.6, 0.62 ){\tiny $4$}
\rput( 4.3, -1.6 ){\tiny $7$}
\rput(  6.67, -1.6  ){\tiny $14$}
\rput( 7.38, 0.62  ){\tiny $16$}

\rput(  5.5, 1.1  ){\tiny $12$}
\rput(  4.45384 , 0.339919  ){\tiny $11$}
\rput(  4.85344 , -0.889919  ){\tiny $9$}
\rput(  6.14656, -0.889919  ){\tiny $3$}
\rput( 6.54616, 0.339919  ){\tiny $6$}


\pscircle[linewidth=0.013,dimen=outer](  0, -3.0  ){0.2}                     
\pscircle[linewidth=0.013,dimen=outer](  -0.722, -3.14 ){0.2}            
\pscircle[linewidth=0.013,dimen=outer](  -1.35, -3.52  ){0.2}
\pscircle[linewidth=0.013,dimen=outer](  -1.79, -4.11   ){0.2}
\pscircle[linewidth=0.013,dimen=outer]( -1.99, -4.82 ){0.2}

\pscircle[linewidth=0.013,dimen=outer]( -1.92, -5.55  ){0.2}                
\pscircle[linewidth=0.013,dimen=outer](  -1.60, -6.21   ){0.2}
\pscircle[linewidth=0.013,dimen=outer](  -1.05, -6.70   ){0.2}
\pscircle[linewidth=0.013,dimen=outer]( -0.367, -6.97  ){0.2}
\pscircle[linewidth=0.013, dimen=outer](  0.367, -6.97  ){0.2}

\pscircle[linewidth=0.013,dimen=outer]( 1.05, -6.70  ){0.2}       
\pscircle[linewidth=0.013,dimen=outer]( 1.60, -6.21  ){0.2}
\pscircle[linewidth=0.013,dimen=outer](  1.92, -5.55   ){0.2}
\pscircle[linewidth=0.013,dimen=outer]( 1.99, -4.82  ){0.2}
\pscircle[linewidth=0.013, dimen=outer]( 1.79, -4.11   ){0.2}

\pscircle[linewidth=0.013,dimen=outer]( 1.35, -3.52 ){0.2}       
\pscircle[linewidth=0.013,dimen=outer]( 0.722, -3.14  ){0.2}

\rput(  0, -3.0 ){\tiny $0$}
\rput(  -0.722, -3.14 ){\tiny $1$}
\rput( -1.35, -3.52  ){\tiny $2$}
\rput( -1.79, -4.11  ){\tiny $3$}
\rput( -1.99, -4.82   ){\tiny $4$}

\rput(  -1.92, -5.55 ){\tiny $5$}
\rput(  -1.60, -6.21 ){\tiny $6$}
\rput(  -1.05, -6.70 ){\tiny $7$}
\rput( -0.367, -6.97  ){\tiny $8$}
\rput(  0.367, -6.97   ){\tiny $9$}

\rput(  1.05, -6.70  ){\tiny $10$}
\rput(  1.60, -6.21 ){\tiny $11$}
\rput(  1.92, -5.55 ){\tiny $12$}
\rput(  1.99, -4.82  ){\tiny $13$}
\rput(  1.79, -4.11   ){\tiny $14$}

\rput(  1.35, -3.52  ){\tiny $15$}
\rput( 0.722, -3.14   ){\tiny $16$}


\psline[linewidth=0.013cm]( -0.202295, -3.03782 )  ( -0.520188, -3.09724  )
\psline[linewidth=0.013cm]( -0.897458, -3.24339 )  ( -1.17242, -3.41364 )
\psline[linewidth=0.013cm]( -1.47141, -3.68621 )  ( -1.6663, -3.94429  )
\psline[linewidth=0.013cm](  -1.84665, -4.30647 )  ( -1.93515, -4.61752  )
\psline[linewidth=0.013cm](  -1.97248, -5.02038 )  (  -1.94264, -5.3424   )

\psline[linewidth=0.013cm]( -1.83192, -5.73155 )  (  -1.68777, -6.02105 )
\psline[linewidth=0.013cm](  -1.44395, -6.34392 )  ( -1.20495, -6.56179 )
\psline[linewidth=0.013cm]( -0.860962, -6.77478 )  ( -0.559401, -6.8916 )
\psline[linewidth=0.013cm](  -0.1617, -6.96595 )  (  0.1617, -6.96595  )
\psline[linewidth=0.013cm](  0.559401, -6.8916 )  ( 0.860962, -6.77478 )

\psline[linewidth=0.013cm](1.20495, -6.56179 )  (  1.44395, -6.34392 )
\psline[linewidth=0.013cm]( 1.68777, -6.02105 )  ( 1.83192, -5.73155 )
\psline[linewidth=0.013cm]( 1.94264, -5.3424 )  ( 1.97248, -5.02038 )
\psline[linewidth=0.013cm](  1.93515, -4.61752 )  ( 1.84665, -4.30647  )
\psline[linewidth=0.013cm]( 1.6663, -3.94429 )  ( 1.47141, -3.68621 )

\psline[linewidth=0.013cm](  1.17242, -3.41364 )  ( 0.897458, -3.24339  )
\psline[linewidth=0.013cm]( 0.520, -3.09724 )  ( 0.202295, -3.03782 )


\psline[linewidth=0.013cm](  0.115419, -3.15284 )  ( 1.80823, -5.39449 )
\psline[linewidth=0.013cm]( -1.74506, -5.61651 )  ( 0.874273, -6.63125 )
\psline[linewidth=0.013cm]( 1.32972, -3.71269 )  ( 1.07054, -6.50973 )
\psline[linewidth=0.013cm]( -1.70496, -4.27997 )  ( -0.452869, -6.7945 )
\psline[linewidth=0.013cm]( 1.84993, -4.944 )  ( -0.225961, -6.83692 )
\psline[linewidth=0.013cm]( -0.774896, -3.31927 )  ( -1.54362, -6.02106 )
\psline[linewidth=0.013cm]( -1.82863, -4.71464 )  ( 0.559646, -3.23588 )
\psline[linewidth=0.013cm]( -1.15913, -3.55717 )  ( 1.60206, -4.07333 )

\psline[linewidth=0.013cm]( -1.20022, -3.65615 )  (1.44886, -6.0711 )
\psline[linewidth=0.013cm]( -1.7279, -5.51073)  (1.79571, -4.852 )
\psline[linewidth=0.013cm]( 0.018375, -3.1983 )  (0.349124, -6.76765 )
\psline[linewidth=0.013cm]( -0.932852, -6.54151 )  (1.22738, -3.68 )


\pscircle[linewidth=0.013,dimen=outer](  5.5, -3.0  ){0.2}                     
\pscircle[linewidth=0.013,dimen=outer](  4.778, -3.14 ){0.2}            
\pscircle[linewidth=0.013,dimen=outer](  4.15, -3.52  ){0.2}
\pscircle[linewidth=0.013,dimen=outer](  3.71, -4.11   ){0.2}
\pscircle[linewidth=0.013,dimen=outer]( 3.51, -4.82 ){0.2}

\pscircle[linewidth=0.013,dimen=outer]( 3.58, -5.55  ){0.2}                
\pscircle[linewidth=0.013,dimen=outer](  3.9, -6.21   ){0.2}
\pscircle[linewidth=0.013,dimen=outer](  4.45, -6.70   ){0.2}
\pscircle[linewidth=0.013,dimen=outer]( 5.133, -6.97  ){0.2}
\pscircle[linewidth=0.013, dimen=outer](  5.867, -6.97  ){0.2}

\pscircle[linewidth=0.013,dimen=outer]( 6.55, -6.70  ){0.2}       
\pscircle[linewidth=0.013,dimen=outer]( 7.1, -6.21  ){0.2}
\pscircle[linewidth=0.013,dimen=outer]( 7.42, -5.55   ){0.2}
\pscircle[linewidth=0.013,dimen=outer]( 7.49, -4.82  ){0.2}
\pscircle[linewidth=0.013, dimen=outer]( 7.29, -4.11   ){0.2}

\pscircle[linewidth=0.013,dimen=outer]( 6.85, -3.52 ){0.2}       
\pscircle[linewidth=0.013,dimen=outer]( 6.22, -3.14  ){0.2}

\rput(  5.5, -3.0 ){\tiny $13$}
\rput(  4.778, -3.14 ){\tiny $16$}
\rput( 4.15, -3.52  ){\tiny $10$}
\rput( 3.71, -4.11  ){\tiny $7$}
\rput( 3.51, -4.82   ){\tiny $3$}

\rput(  3.58, -5.55 )   {\tiny $0$}
\rput(  3.9, -6.21 )   {\tiny $14$}
\rput(  4.45, -6.70 )   {\tiny $1$}
\rput( 5.133, -6.97  )   {\tiny $5$}
\rput(  5.867, -6.97   )  {\tiny $9$}

\rput(  6.55, -6.70  ){\tiny $2$}
\rput(  7.1, -6.21 ){\tiny $4$}
\rput( 7.42, -5.55 ){\tiny $11$}
\rput(  7.49, -4.82  ){\tiny $15$}
\rput(  7.29, -4.11   ){\tiny $12$}

\rput(  6.85, -3.52  ){\tiny $8$}
\rput( 6.222, -3.14   ){\tiny $6$}


\psline[linewidth=0.013cm]( 5.3, -3.03782 )  ( 4.98, -3.09724  )
\psline[linewidth=0.013cm](4.60, -3.24339 )  ( 4.327, -3.41364 )
\psline[linewidth=0.013cm](4.028, -3.68621 )  ( 3.8337, -3.94429  )
\psline[linewidth=0.013cm](  3.65335, -4.30647 )  (3.565, -4.61752  )
\psline[linewidth=0.013cm](  3.52752, -5.02038 )  (  3.55736, -5.3424   )

\psline[linewidth=0.013cm]( 3.668, -5.73155 )  (  3.812, -6.02105 )
\psline[linewidth=0.013cm](  4.056, -6.34392 )  ( 4.29, -6.56179 )
\psline[linewidth=0.013cm]( 4.639, -6.77478 )  ( 4.94, -6.8916 )
\psline[linewidth=0.013cm](  5.338, -6.96595 )  (  5.662, -6.96595  )
\psline[linewidth=0.013cm](  6.06, -6.8916 )  ( 6.36, -6.77478 )

\psline[linewidth=0.013cm](6.7, -6.56179 )  (  6.94, -6.34392 )
\psline[linewidth=0.013cm]( 7.18, -6.02105 )  ( 7.33, -5.73155 )
\psline[linewidth=0.013cm]( 7.442, -5.3424 )  ( 7.47, -5.02038 )
\psline[linewidth=0.013cm](  7.435, -4.61752 )  ( 7.346, -4.30647  )
\psline[linewidth=0.013cm]( 7.166, -3.94429 )  ( 6.97, -3.68621 )

\psline[linewidth=0.013cm](  6.672, -3.41364 )  ( 6.397, -3.24339  )
\psline[linewidth=0.013cm]( 6.02, -3.09724 )  ( 5.7023, -3.03782 )

\psline[linewidth=0.013cm](  5.615, -3.15284 )  ( 7.308, -5.39449 )
\psline[linewidth=0.013cm]( 3.755, -5.61651 )  ( 6.374, -6.63125 )
\psline[linewidth=0.013cm]( 6.829, -3.71269 )  ( 6.57, -6.50973 )
\psline[linewidth=0.013cm]( 3.795, -4.27997 )  (5.047, -6.7945 )
\psline[linewidth=0.013cm]( 7.349, -4.944 )  ( 5.274, -6.83692 )
\psline[linewidth=0.013cm]( 4.725, -3.31927 )  ( 3.956, -6.02106 )
\psline[linewidth=0.013cm]( 3.67, -4.71464 )  ( 6.059, -3.23588 )
\psline[linewidth=0.013cm]( 4.34, -3.55717 )  ( 7.1, -4.07333 )

\psline[linewidth=0.013cm]( 4.299, -3.65615 )  (6.948, -6.0711 )
\psline[linewidth=0.013cm]( 3.772, -5.51073)  (7.295, -4.852 )
\psline[linewidth=0.013cm]( 5.518, -3.1983 )  (5.849, -6.76765 )
\psline[linewidth=0.013cm]( 4.567, -6.54151 )  (6.727, -3.68 )

\end{pspicture}}
\caption{\small   The graphs $H_i$ and $G_i$ for $i=0,1$ where $q=17$.
\label{fig:caseq17}}
\end{center}
\end{figure}

\item For $q=19$:

Let $\mathbb{Z}_{19}$ be a finite field of order $19$ and consider the graphs
$H_0$, $H_1$, $G_0$ and $G_1$ showed in Figure~\ref{fig:caseq19}.
  The graphs $G_0$ and $G_1$ are isomorphic, have order 19, girth 5, the vertices of the set $S= \{0,2,3,5,6,12,13,16,17 \}$  have degree 4 and  the other ones have degree 3. The weights or Cayley colors modulo $19$ of $G_0$ (and $G_1$) are $ \pm \{ 1, 4, 7, 8 \}$ (and $ \pm \{  2, 3, 5, 6, 9 \}$, respectively).
Regarding $H_0$ and $H_1$,  they are labeled with the elements of $\mathbb{Z}_{19}- S$ and verify   $E(H_0) \cap  E(H_1) = \emptyset$, $E(H_0) \cap  E(G_1) = \emptyset$ and
$E(H_1) \cap  E(G_0) = \emptyset$.




\begin{figure}[h!]
\begin{center}
\scalebox{1.05}
{\begin{pspicture}(1.9, -7.2 )(3., 1.40)

\definecolor{color0}{rgb}{0.0,0.,0.}

\rput(  0.,  2.4  ){\tiny$H_0$}
\rput( 0,   -2.6   ){\tiny $G_0$}
\rput(  5.5,  2.4  ){\tiny$H_1$}
\rput( 5.5,   -2.6   ){\tiny $G_1$}

\pscircle[linewidth=0.013,dimen=outer](0., 2.){0.2}
\pscircle[linewidth=0.013,dimen=outer](-1.9, 0.62){0.2}
\pscircle[linewidth=0.013,dimen=outer](-1.2, -1.6){0.2}
\pscircle[linewidth=0.013,dimen=outer]( 1.2, -1.6 ){0.2}
\pscircle[linewidth=0.013,dimen=outer]( 1.9, 0.62 ){0.2}

\pscircle[linewidth=0.013,dimen=outer]( 0., 1.1 ){0.2}
\pscircle[linewidth=0.013,dimen=outer]( -1.04, 0.34 ){0.2}
\pscircle[linewidth=0.013,dimen=outer]( -0.65, -0.89 ){0.2}
\pscircle[linewidth=0.013,dimen=outer]( 0.65, -0.89 ){0.2}
\pscircle[linewidth=0.013, dimen=outer]( 1.05, 0.34 ){0.2}


\psline[linewidth=0.013cm]( -0.17119, 1.87562 )  (  -1.73092, 0.742411  )
\psline[linewidth=0.013cm]( -1.83672, 0.416788 )  (-1.24096, -1.41679)
\psline[linewidth=0.013cm]( -0.963968, -1.61803 )  (0.963968, -1.61803  )
\psline[linewidth=0.013cm]( 1.24096, -1.41679 )  ( 1.83672, 0.416788  )
\psline[linewidth=0.013cm]( 0.17119, 1.87562 )  ( 1.73092, 0.742411  )


\psline[linewidth=0.013cm](-0.0581907, 0.920907 )  ( -0.588373, -0.710826 )
\psline[linewidth=0.013cm]( -0.893817, 0.229233 )  ( 0.494218, -0.779233 )
\psline[linewidth=0.013cm]( -0.857853, 0.339919 )  ( 0.857853, 0.339919  )
\psline[linewidth=0.013cm](  -0.494218, -0.779233 )  ( 0.893817, 0.229233  )
\psline[linewidth=0.013cm]( 0.0581907, 0.920907 )  (0.588373, -0.710826)


\psline[linewidth=0.013cm]( 0,  1.8 )( 0,1.3 )
\psline[linewidth=0.013cm]( -1.73092, 0.562411 ) ( -1.21735, 0.395542 )
\psline[linewidth=0.013cm]( 1.73092, 0.562411  )( 1.21735, 0.395542 )
\psline[linewidth=0.013cm]( 1.06977, -1.47241 )( 0.752365, -1.03554 )
\psline[linewidth=0.013cm]( -1.06977, -1.47241 )( -0.752365, -1.03554 )

\rput( 0, 2.0 ){\tiny $7$}
\rput(-1.9, 0.62 ){\tiny $8$}
\rput( -1.2, -1.6 ){\tiny $9$}
\rput(  1.2, -1.6  ){\tiny $10$}
\rput( 1.9, 0.62  ){\tiny $11$}

\rput(  0., 1.1  ){\tiny $15$}
\rput(  -1.04616, 0.339919  ){\tiny $4$}
\rput(  -0.646564, -0.889919  ){\tiny $1$}
\rput(  0.646564, -0.889919  ){\tiny $14$}
\rput( 1.04616, 0.339919  ){\tiny $18$}

\pscircle[linewidth=0.013,dimen=outer](  5.5, 2.  ){0.2}
\pscircle[linewidth=0.013,dimen=outer]( 3.6, 0.62){0.2}
\pscircle[linewidth=0.013,dimen=outer]( 4.3, -1.6){0.2}
\pscircle[linewidth=0.013,dimen=outer]( 6.7, -1.6 ){0.2}
\pscircle[linewidth=0.013,dimen=outer]( 7.4, 0.62 ){0.2}

\pscircle[linewidth=0.013,dimen=outer]( 5.5, 1.1 ){0.2}
\pscircle[linewidth=0.013,dimen=outer]( 4.46, 0.34 ){0.2}
\pscircle[linewidth=0.013,dimen=outer]( 4.85, -0.89 ){0.2}
\pscircle[linewidth=0.013,dimen=outer]( 6.15, -0.89 ){0.2}
\pscircle[linewidth=0.013, dimen=outer]( 6.55, 0.34 ){0.2}


\psline[linewidth=0.013cm]( 5.32881, 1.87562 )  ( 3.76908, 0.742411  )
\psline[linewidth=0.013cm]( 3.66328, 0.416788 )  (4.25904, -1.41679)
\psline[linewidth=0.013cm]( 4.53603, -1.61803 )  (6.46397, -1.61803  )
\psline[linewidth=0.013cm]( 6.74096, -1.41679 )  ( 7.33672, 0.416788  )
\psline[linewidth=0.013cm]( 5.67119, 1.87562 )  ( 7.23, 0.742411   )


\psline[linewidth=0.013cm]( 5.44181, 0.920907 )  (  4.91163 , -0.710826 )
\psline[linewidth=0.013cm](  4.6  , 0.229233 )  ( 5.99422, -0.779233 )
\psline[linewidth=0.013cm](  4.64215  , 0.339919 )  ( 6.35785 , 0.339919  )
\psline[linewidth=0.013cm](   5.00578  , -0.779233 )  ( 6.39382 , 0.229233  )
\psline[linewidth=0.013cm]( 5.5581907, 0.920907 )  (6.08837  , -0.710826)


\psline[linewidth=0.013cm]( 5.5,  1.8 )( 5.5, 1.3 )
\psline[linewidth=0.013cm]( 3.76908, 0.562411 ) (4.28265,  0.395542 )
\psline[linewidth=0.013cm]( 7.2309, 0.562411  )( 6.71735, 0.395542 )
\psline[linewidth=0.013cm]( 6.56977, -1.47241 )( 6.25237, -1.03554 )
\psline[linewidth=0.013cm]( 4.43023, -1.47241 )( 4.74764, -1.03554 )

\rput( 5.5, 2.0 ){\tiny $14$}
\rput(3.6, 0.62 ){\tiny $18$}
\rput( 4.3, -1.6 ){\tiny $15$}
\rput(  6.67, -1.6  ){\tiny $10$}
\rput( 7.38, 0.62  ){\tiny $8$}

\rput(  5.5, 1.1  ){\tiny $1$}
\rput(  4.45384 , 0.339919  ){\tiny $9$}
\rput(  4.85344 , -0.889919  ){\tiny $4$}
\rput(  6.14656, -0.889919  ){\tiny $7$}
\rput( 6.54616, 0.339919  ){\tiny $11$}


\pscircle[linewidth=0.013,dimen=outer](  0, -3.0  ){0.2}                     
\pscircle[linewidth=0.013,dimen=outer]( -0.65, -3.1  ){0.2}            
\pscircle[linewidth=0.013,dimen=outer]( -1.23, -3.42 ){0.2}
\pscircle[linewidth=0.013,dimen=outer]( -1.67, -3.91  ){0.2}
\pscircle[linewidth=0.013,dimen=outer](  -1.94, -4.51  ){0.2}

\pscircle[linewidth=0.013,dimen=outer](  -1.99, -5.17   ){0.2}                
\pscircle[linewidth=0.013,dimen=outer](  -1.83, -5.80    ){0.2}
\pscircle[linewidth=0.013,dimen=outer]( -1.47, -6.35   ){0.2}
\pscircle[linewidth=0.013,dimen=outer](   -0.952, -6.76  ){0.2}
\pscircle[linewidth=0.013, dimen=outer](   -0.329, -6.97  ){0.2}

\pscircle[linewidth=0.013,dimen=outer]( 0.329, -6.97    ){0.2}       
\pscircle[linewidth=0.013,dimen=outer](0.952, -6.76    ){0.2}
\pscircle[linewidth=0.013,dimen=outer](  1.47, -6.35    ){0.2}
\pscircle[linewidth=0.013,dimen=outer](  1.83, -5.80  ){0.2}
\pscircle[linewidth=0.013, dimen=outer](1.99, -5.17    ){0.2}

\pscircle[linewidth=0.013,dimen=outer](  1.94, -4.51   ){0.2}       
\pscircle[linewidth=0.013,dimen=outer]( 1.67, -3.91  ){0.2}
\pscircle[linewidth=0.013,dimen=outer](  1.23, -3.42                        ){0.2}
\pscircle[linewidth=0.013,dimen=outer](  0.649, -3.11  ){0.2}

\rput(  0, -3.0 ){\tiny $0$}
\rput( -0.649, -3.11 ){\tiny $1$}
\rput( -1.23, -3.42   ){\tiny $2$}
\rput(  -1.67, -3.91   ){\tiny $3$}
\rput(  -1.94, -4.51   ){\tiny $4$}

\rput( -1.99, -5.17  ){\tiny $5$}
\rput(  -1.83, -5.80  ){\tiny $6$}
\rput( -1.47, -6.35  ){\tiny $7$}
\rput(  -0.952, -6.76  ){\tiny $8$}
\rput(  -0.329, -6.97    ){\tiny $9$}

\rput(   0.329, -6.97  ){\tiny $10$}
\rput( 0.952, -6.76  ){\tiny $11$}
\rput(  1.47, -6.35  ){\tiny $12$}
\rput( 1.83, -5.80  ){\tiny $13$}
\rput(  1.99, -5.17   ){\tiny $14$}

\rput(  1.94, -4.51  ){\tiny $15$}
\rput( 1.67, -3.91  ){\tiny $16$}
\rput( 1.23, -3.42   ){\tiny $17$}
\rput( 0.649, -3.11   ){\tiny $18$}


\psline[linewidth=0.013cm]( -0.474061, -3.07911)  ( -0.175338, -3.02926 )
\psline[linewidth=0.013cm]( -1.072, -3.337 )  ( -0.8, -3.19 )
\psline[linewidth=0.013cm](  -1.559, -3.779 )  (  -1.349, -3.5525 )
\psline[linewidth=0.013cm]( -1.86739, -4.346)  (-1.74574, -4.068 )
\psline[linewidth=0.013cm]( -1.97849, -4.988)  (-1.95348, -4.68 )

\psline[linewidth=0.013cm]( -1.87518, -5.631)  ( -1.94953, -5.337 )
\psline[linewidth=0.013cm]( -1.56867, -6.2 )  (  -1.73432, -5.952  )
\psline[linewidth=0.013cm](  -1.09217, -6.649 )  (  -1.33117, -6.463 )
\psline[linewidth=0.013cm]( -0.49732, -6.915)  ( -0.783764, -6.8166 )
\psline[linewidth=0.013cm]( 0.151427, -6.972 )  ( -0.151427, -6.97  )

\psline[linewidth=0.013cm](  0.783764, -6.816 )  (  0.49732, -6.915  )
\psline[linewidth=0.013cm]( 1.33117, -6.463 )  (1.09217, -6.649  )
\psline[linewidth=0.013cm](  1.73432, -5.95  )  ( 1.56867, -6.205  )
\psline[linewidth=0.013cm](  1.94953, -5.337 )  ( 1.87518, -5.631  )
\psline[linewidth=0.013cm](  1.95348, -4.686 )  ( 1.97849, -4.988  )

\psline[linewidth=0.013cm](  1.74574, -4.0689  )  ( 1.86739, -4.346  )
\psline[linewidth=0.013cm](  1.34882, -3.552 )  ( 1.55394, -3.775  )
\psline[linewidth=0.013cm](  0.805736, -3.1929  )  ( 1.07209, -3.337   )
\psline[linewidth=0.013cm](  0.175338, -3.029 )  ( 0.474061, -3.079  )

\psline[linewidth=0.013cm](  1.05645, -3.4217 )  (  -1.05645, -3.4217 )
\psline[linewidth=0.013cm]( -1.7837, -4.388  )  ( -0.155104, -3.12  )
\psline[linewidth=0.013cm](  1.7837, -4.388 )  (  0.155104, -3.12 )
\psline[linewidth=0.013cm](  1.27668, -3.6122 )  (1.7833, -5.6128 )
\psline[linewidth=0.013cm]( -1.65869, -5.8969 )  ( 0.15633, -6.879 )

\psline[linewidth=0.013cm](  -1.7833, -5.6128 )  (  -1.27668, -3.6122  )
\psline[linewidth=0.013cm]( 1.65869, -5.8969 )  ( -0.15633, -6.879  )
\psline[linewidth=0.013cm]( 1.49096, -3.969  )  (  -1.80979, -5.102  )
\psline[linewidth=0.013cm]( 1.54302, -4.048 )  (  -0.820583, -6.616  )

\psline[linewidth=0.013cm](  1.80979, -5.102 )  ( -1.49096, -3.969 )
\psline[linewidth=0.013cm](  0.543357, -3.27 )  ( -1.36541, -6.192  )
\psline[linewidth=0.013cm](  1.36541, -6.1922  )  ( -0.543357, -3.27 )

\psline[linewidth=0.013cm](  0.820583, -6.616 )  (-1.54302, -4.0487  )
\psline[linewidth=0.013cm]( 1.29822, -6.295 )  ( -1.81994, -5.2246  )


\pscircle[linewidth=0.013,dimen=outer](  5.5, -3.0  ){0.2}                     
\pscircle[linewidth=0.013,dimen=outer]( 4.85, -3.1  ){0.2}            
\pscircle[linewidth=0.013,dimen=outer](  4.27 , -3.42 ){0.2}
\pscircle[linewidth=0.013,dimen=outer]( 3.83, -3.91  ){0.2}
\pscircle[linewidth=0.013,dimen=outer](  3.56,   -4.51  ){0.2}

\pscircle[linewidth=0.013,dimen=outer](   3.51 , -5.17   ){0.2}                
\pscircle[linewidth=0.013,dimen=outer](   3.67 , -5.80    ){0.2}
\pscircle[linewidth=0.013,dimen=outer](   4.03 , -6.35   ){0.2}
\pscircle[linewidth=0.013,dimen=outer](   4.548 , -6.76  ){0.2}
\pscircle[linewidth=0.013, dimen=outer](    5.171  , -6.97  ){0.2}

\pscircle[linewidth=0.013,dimen=outer](  5.829 , -6.97    ){0.2}       
\pscircle[linewidth=0.013,dimen=outer](  6.452 , -6.76    ){0.2}
\pscircle[linewidth=0.013,dimen=outer](  6.97  , -6.35    ){0.2}
\pscircle[linewidth=0.013,dimen=outer](  7.33 , -5.80  ){0.2}
\pscircle[linewidth=0.013, dimen=outer](  7.49  , -5.17    ){0.2}

\pscircle[linewidth=0.013,dimen=outer](  7.44 , -4.51   ){0.2}       
\pscircle[linewidth=0.013,dimen=outer](  7.17  , -3.91  ){0.2}
\pscircle[linewidth=0.013,dimen=outer](  6.73 , -3.42                        ){0.2}
\pscircle[linewidth=0.013,dimen=outer](  6.149, -3.11  ){0.2}

\rput(  5.5 , -3.0 ){\tiny $5$}
\rput( 4.851 , -3.11 ){\tiny $7$}
\rput(   4.27, -3.42   ){\tiny $16$}
\rput(  3.83 , -3.91   ){\tiny $13$}
\rput(   3.56, -4.51   ){\tiny $10$}

\rput(  3.51, -5.17  ){\tiny $0$}
\rput(  3.67 , -5.80  ){\tiny $6$}
\rput(   4.03, -6.35  ){\tiny $11$}
\rput(   4.548, -6.76  ){\tiny $1$}
\rput(  5.171 , -6.97    ){\tiny $14$}

\rput(  5.829 , -6.97  ){\tiny $9$}
\rput(   6.452 , -6.76  ){\tiny $4$}
\rput(  6.97 , -6.35  ){\tiny $2$}
\rput(  7.33 , -5.80  ){\tiny $12$}
\rput(  7.49 , -5.17   ){\tiny $18$}

\rput(   7.44 , -4.51  ){\tiny $15$}
\rput(  7.17, -3.91  ){\tiny $17$}
\rput(   6.73 , -3.42   ){\tiny $3$}
\rput(  6.149 , -3.11   ){\tiny $8$}


\psline[linewidth=0.013cm](  5.0259 , -3.07911)  ( 5.32466 , -3.02926 )
\psline[linewidth=0.013cm](  4.428 , -3.337 )  (  4.7, -3.19 )
\psline[linewidth=0.013cm](  3.94  , -3.779 )  (  4.151 , -3.5525 )
\psline[linewidth=0.013cm](   3.6326, -4.346)  ( 3.75426 , -4.068 )
\psline[linewidth=0.013cm](   3.5215 , -4.988)  (  3.54652, -4.68 )

\psline[linewidth=0.013cm](  3.6248, -5.631)  (  3.55047 , -5.337 )
\psline[linewidth=0.013cm](  3.93133 , -6.2 )  (  3.76568  , -5.952  )
\psline[linewidth=0.013cm](  4.4078 , -6.649 )  (  4.16883, -6.463 )
\psline[linewidth=0.013cm](  5.0026, -6.915)  (  4.71624, -6.8166 )
\psline[linewidth=0.013cm](  5.6514 , -6.972 )  ( 5.34857 , -6.97  )

\psline[linewidth=0.013cm](  6.2837 , -6.816 )  (   5.99732 , -6.915  )
\psline[linewidth=0.013cm](   6.83117 , -6.463 )  ( 6.59217 , -6.649  )
\psline[linewidth=0.013cm](  7.23432 , -5.95  )  (  7.05394 , -6.205  )
\psline[linewidth=0.013cm](  7.44953 , -5.337 )  ( 7.37518 , -5.631  )
\psline[linewidth=0.013cm](  7.45348  , -4.686 )  (  7.47849 , -4.988  )

\psline[linewidth=0.013cm](  7.24574 , -4.0689  )  ( 7.36739 , -4.346  )
\psline[linewidth=0.013cm](  6.84882 , -3.552 )  (  7.05394 , -3.775  )
\psline[linewidth=0.013cm](  6.30574, -3.1929  )  (  6.57209 , -3.337   )
\psline[linewidth=0.013cm](  5.67534 , -3.029 )  (  5.974 , -3.079  )

\psline[linewidth=0.013cm](   6.55645 , -3.4217 )  (  4.44355, -3.4217 )
\psline[linewidth=0.013cm](  3.7163  , -4.388  )  (  5.3449 , -3.12  )
\psline[linewidth=0.013cm](  7.2837  , -4.388 )  (  5.655104, -3.12 )
\psline[linewidth=0.013cm](   6.77668 , -3.6122 )  (  7.2833, -5.6128 )
\psline[linewidth=0.013cm](   3.8413, -5.8969 )  ( 5.65633, -6.879 )

\psline[linewidth=0.013cm](  3.7167 , -5.6128 )  (  4.22332 , -3.6122  )
\psline[linewidth=0.013cm](  7.15869 , -5.8969 )  (  5.34367, -6.879  )
\psline[linewidth=0.013cm]( 6.99096, -3.969  )  (  3.69021, -5.102  )
\psline[linewidth=0.013cm]( 7.04302 , -4.048 )  (  4.67942 , -6.616  )

\psline[linewidth=0.013cm](  7.30979 , -5.102 )  ( 4.00904 , -3.969 )
\psline[linewidth=0.013cm](  6.04336, -3.27 )  (   4.13459, -6.192  )
\psline[linewidth=0.013cm](  6.86541 , -6.1922  )  ( 4.95664 , -3.27 )

\psline[linewidth=0.013cm](   6.32058 , -6.616 )  ( 3.95698 , -4.0487  )
\psline[linewidth=0.013cm](  6.79822, -6.295 )  ( 3.68006 , -5.2246  )

\end{pspicture}}
\caption{\small   The graphs $H_i$ and $G_i$ for $i=0,1$ where  $q=19$.
\label{fig:caseq19}}
\end{center}
\end{figure}
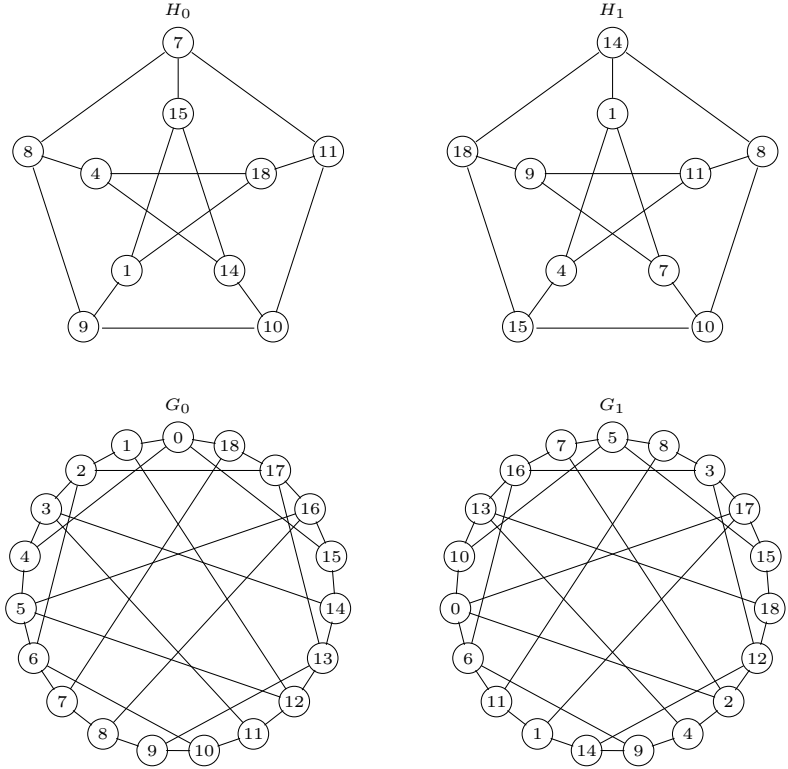

\end{itemize}


\noindent In the next result we apply Theorem \ref{teoC} to~$q \in \{16, 17, 19 \}$. The obtained  graph $\mathcal{C}_q(H_0,H_1,G_0,G_1)$ is a $(q+3,5)$-regular graph with less vertices than any other  $(q+3)$-regular graph of girth 5 so far known, and therefore  the upper bound $rec(k,5)$ for $k \in \{19, 20,22\}$ is improved.
As it is explained in the so-called Reduction 2  in \cite{AABL12},  referred as ``Deletion''   in \cite{F09},  by removing pairs of blocks $P_x$ and $L_m$  from $    \mathcal{C}_q (H_0,H_1,G_0,G_1)$, we also generate new graphs which improve $rec(k,5)$ for $k \in \{17, 18,21\}$.
\\

\begin{theorem}
The following upper bound $rec(k,5)$ on the order $n(k,5)$  of a $k$-regular cage of girth $5$ holds
{
\begin{center}
\begin{tabular}{|c||@{$\,$}c@{$\,$}|}
\hline
 $k$ &  \  $rec(k,5)$ \ \\
       \hline \hline
      17   &   436  \\
       \hline
     18  &  468   \\
  \hline
     19  &  500   \\
  \hline
     20 &  564   \\
  \hline
     21  &  666   \\
  \hline
     22  &  704   \\
       \hline
\end{tabular}
\end{center}}

\end{theorem}

\begin{proof}
Using the  graphs given in Construction 1, we obtain for $q\in\{16,17,19\}$ the graph $\mathcal{C}_{q} ( H_0, H_1,G_0,G_1)$ as in Theorem \ref{teoC},  which has girth $5$. Moreover we have the following considerations:

\noindent For $q=16$, $\mathcal{C}_{16} (H_0, H_1,G_0,G_1)$ is a $(19,5)$-graph of order $2 \cdot 16^2 - 12= 500$ implying that any $(19,5)$-cage has at most $500$ vertices.
Removing of $\mathcal{C}_{16} ( H_0, H_1,G_0,G_1)$ (using the operation called ``Reduction 2"  in \cite{AABL12}) a block of lines $L_{m}$ and a block of points $P_{x}$, for $x, m \in (\mathbb{Z}_2)^4- \{ 0000 \}$, we construct a $18$-regular graph with $500-2 \cdot 16=468$ vertices.  Similarly, deleting from this last graph another pair of blocks we obtain a  $17$-regular graph of girth~$5$ with $436$ vertices. Each of these $k$-regular graphs ($k= 17,18,19$) has $12 $ vertices less than the ones constructed by Schwenk in~\cite{S08}.

 \noindent For $q=17$, $\mathcal{C}_{17}  (H_0, H_1,G_0,G_1)$ is a $(20,5)$-graph of order $2 \cdot 17^2 - 14= 564$, which implies that a $(20,5)$-cage has at most $564$ vertices.

\noindent For $q=19$, $ \mathcal{C}_{19} ( H_0, H_1,G_0,G_1)$ is a $(22,5)$-graph of order $2 \cdot 19^2 - 18= 704$, which also implies that  any $(22,5)$-cage has at most $704$ vertices.
Newly, deleting any block of points and any block of lines (except $P_{0}$ and $L_{0}$ blocks), it is straightforward to check out that  $n(21,5) \le 666$.
\end{proof}

\begin{remark}
It is important to note that the construction of a $(q+3)$-regular graph of girth at least $5$  using  bi-regular amalgams into a subgraph of $C_q$ involves the existence of two $3$-regular graphs $H_0$ and $H_1$ and two $(3,4)$-regular graphs $G_0$ and $G_1$ all of them with girth at least $5$. The  graph
$\mathcal{C}_q (H_0,H_1, G_0,G_1)$  has order $2 (q^2 - (q-n(H_0)) \ge 2 (q^2 - q +10).$ It means that  our construction is the best possible one for $q=16$ and $q= 17$, because a $4$-regular amalgam could only be possible for $q \ge n(4,5)=19$ (recall that the $(4,5)$-cage is the Robertson Graph that has order $19$).
\end{remark}

\section {Amalgamating into elliptic semiplanes of type $L$}

In this section we use the techniques given by Funk in~\cite{F09} to amalgamate a pair of suitable regular graphs into the Levi graph of an elliptic semiplane of type $L_q$. Recall that the semiplane of type  $\emph{L}$ is obtained by deleting, from the projective, a pair of non incident point and line, all the lines incident with the point and all the points incident with the line. Moreover, the Levi graph, denoted by $L_q$, is bipartite, $q$-regular and has $2(q^2-1)$ vertices of which $q^2-1$ are points and $q^2-1$ are lines in the elliptic semiplane, both partitioned into $q+1$ parallel classes of $q-1$ elements each.


\noindent We divide the section into two parts. First we construct the regular graphs $G_0$, $G_1$ to amalgamate and later we describe the resulting graph $\mathcal{L}_q(G_0,G_1)$.

\subsection {Constructions of regular graphs of girth five}

To apply the Funk's techniques we need to construct two regular graphs with the same order, girth at least five and having disjoint Cayley colors, one of them to amalgamate in the point blocks and the other in the line blocks of $L_q$.

\noindent Let $\mathbb{Z}_n$ be the set of integers modulo $n$, and $J=\{k_1, \ldots, k_w\}\subset \mathbb{Z}_n-0$. Recall that a\emph{  circulant graph $Z_n(J)$} is a graph with vertex set $\mathbb{Z}_n$ and edges $\alpha \beta$ where }
$ \beta - \alpha \in J$. Let $n=2t$ and suppose that every element  of $J$ is odd.  We denote by $ S_{2t}( k_1,   \ldots, k_w) $  the subgraph of the circulant graph $\mathbb{Z}_{2t}( k_1,   \ldots, k_w)$ with edge set $\{  \{  2v, 2v+k_j \} :  0  \le  v  \le t-1,  \; 1  \le  j   \le  w  \}$
where the sum is taken module $2t$. Moreover,  we denote by  $ S_{\infty} ( k_1,   \ldots, k_w) $ the (infinite) graph when $\mathbb{Z}_{2t}=\mathbb{Z}$.

\noindent In the following lemma we describe some relevant properties of this graph:

\begin{lemma}  \label{oddsemicirculant}
Given an integer $t \ge 5 $, and  a sequence $k_1,  \ldots, k_w$  of different odd elements from  $\mathbb{Z}_{2t}$,
 the graph $ S_{2t}( k_1,   \ldots, k_w) $ is  $w$-regular, bipartite and has at most $w$ Cayley colors in $\mathbb{Z}_{2t}$. Moreover, the girth of $ S_{2t}( k_1,   \ldots, k_w) $ is at least~$6$    iff all the numbers $k_i-k_j$ are different for $i \neq j $ and $ 1  \le i,  j   \le  w$. These properties hold for $2t=\infty$.
\end{lemma}

\begin{proof}
Given an odd element $k_j \in \mathbb{Z}_{2t}$, the set of edges $\{  \{  2v,2v+k_j \} :  0  \le  v  \le t-1 \}$ defines a matching between   the  vertices with even label  and the   ones with odd label in $\mathbb{Z}_{2t} $.  Therefore, $G= S_{2t}( k_1,   \ldots, k_w) $ is $w$-regular, bipartite and has even girth $g \ge 4$.

\noindent Assume the numbers $k_i-k_j$ are different for $i \neq j $ and $ 1  \le i,  j   \le  w$. We prove that the girth of $G$ is greater or equal than 6. Suppose that there exists a $4$-cycle $ v_0 v_1  v_2  v_3 v_0$. By reordering,  we may take  $v_0$, $v_2$   even and $v_1$, $v_3$  odd. So, $v_1=v_0+k_i$, $v_2=v_1-k_j$, $v_3=v_2+k_p$, $v_0=v_3-k_q$ with  $i \neq j$, $ p \neq q,   p\neq j,   q \neq i.  $ Then, $k_i-k_j+k_p-k_q=0 $ and $k_i-k_j= k_q-k_p$ in $\mathbb{Z}_{2t}$ which is a contradiction, since by hypothesis all these numbers are different. Hence the girth of $G$ must be at least 6 because it is bipartite. The proof is the same when   $\mathbb{Z}_{2t}= \mathbb{Z}$, taking into account that in this case the equalities are considered in $ \mathbb{Z}.$
\end{proof}

\noindent The $(q+1,6)$-cages, with $q$ a prime power, are  known examples of this type of graphs.  For instance, the Heawood graph can be constructed as $S_{14} (1,-1,5)$. They can also be represented by using  \emph{perfect difference sets}  (see \cite{J05, S38})  and  as the Levi graphs of the projective plane over the field $\mathbb{F}_q$.  

\noindent We can use a graph $ S_{2t}( k_1,   \ldots, k_w) $ with  girth at least $6$  to  construct a new regular graph with the same order, greater degree and girth at least five.

\begin{definition}

Given an   integer $t \ge 5$,   a sequence  of different odd elements $k_1,  \ldots, k_w$ and two different even elements $ 0<P, Q<t $ from  $\mathbb{Z}_{2t}$,
we denote by $ S_{2t}( P, Q;  k_1,   \ldots, k_w  ) $ the graph obtained adding to    $ S_{2t}(  k_1,   \ldots, k_w  ) $  the new  edges  $\{ 2v, 2v+P \}$   and $\{2v+1,2v+1+Q\}$, where the sum is taken modulo 2t.
The graph $ S_{ \infty} ( P, Q;  k_1,   \ldots, k_w  ) $ is defined in a similar way over $ \mathbb{Z}$.
\end{definition}

\noindent Notice that if $P$ divides $2t$ the subgraph of $S_{2t}( P, Q; k_1,   \ldots, k_w) $  induced by the even numbers, is formed by $P/2$ cycles, each of them with size   $2t/P.$ Similar result holds when $Q$ divides ${2t}$ and the subgraph of $S_{2t}( P, Q; k_1,   \ldots, k_w) $ induced by the odd numbers.
The standard Generalized Petersen Graphs  with ${2t}$ vertices introduced by Coxeter in~\cite{C50} are obtained as $S_{2t}( 2,Q; 1)$ and  the $I$-graph $I( t,j,k)$ {in \cite{ZHP09} as $S_{2t}( 2j, 2k;1)$.
Funk uses in \cite{F09} a  $4$-regular generalization $P(k, \eta,  \nu)$ of the Petersen graph which  corresponds to  $S_{2k} (2,2 \eta; 1, 2\nu+1 )$. As illustration, Figure~\ref{cases24and30} (left) depicts $S_{24}( 2,10; 1,7)$ where the highlighted edges have weights $1$, $2$, $10$,  and
Figure~\ref{cases24and30} (right) shows the $(5,5)$-cage or Foster Graph, which corresponds to  $S_{30}(6,12; 1,-1,9) $, where the three  Petersen subgraphs contained in this cage are highlighted.

\begin{figure}[h!]
 \centering
\includegraphics[width=0.7\textwidth]{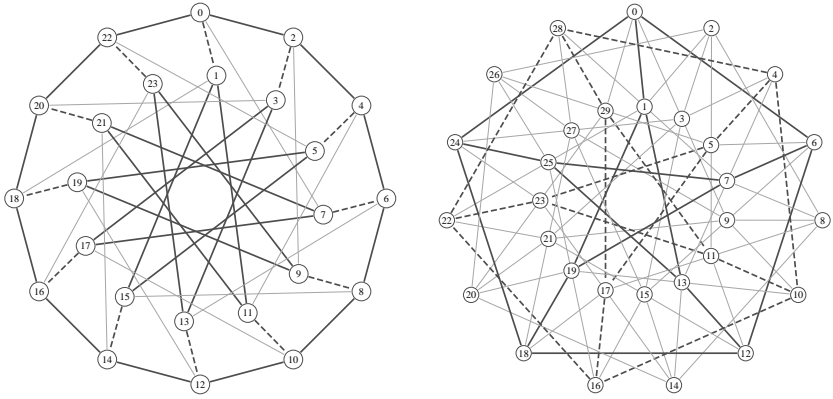}
\caption{\small  The graph $S_{24}( 2,10; 1,7)$ and the $(5,5)$-cage of Foster.}
\label{cases24and30}
\end{figure}

\noindent Next, we summarize some useful properties of these graphs.

\begin{lemma} \label{semicirculant}

The graph $ S_{2t}( P, Q; k_1,   \ldots, k_w) $, defined over  $\mathbb{Z}_{2t},$ is $(w+2)$-regular and has at most $w+2$ Cayley colors. Moreover, the girth of $ S_{2t}( P,Q; k_1,   \ldots, k_w) $ is at least 5 if and only if the following conditions hold:
\begin{itemize}
\item[(i)]  The numbers $3P, 4P, 3Q,  4Q$  are different from $0$ in $\mathbb{Z}_{2t}$.
\item[(ii)]  All the numbers $k_i-k_j$ are different  for $i \neq j $ and $ 1  \le i,  j   \le  w$.
\item[(iii)]  No relation  $k_i-k_j = \omega-\omega '$ holds, for a pair $\omega, \omega ' \in \Omega=\{0, \pm P, \pm Q \}.$
\end{itemize}
The result also holds when $Z_{2t}=Z$.

\end{lemma}

\begin{proof} Denote $ G=S_{2t}( P, Q;  k_1,   \ldots, k_w  ) $. According to Lemma \ref{oddsemicirculant}, the  subgraph  $B= S_{2t}(  k_1,   \ldots, k_w) $ is an $w$-regular bipartite graph with  girth at least~$6$ iff item $(ii)$ is satisfied. The partite sets of $B$ are   the set of even vertices, denoted by $Ev$, and  the set of odd vertices, denoted by $Od$, of $\mathbb{Z}_{2t}$. Consider $T_0$ and $T_1$ the circulant graphs whose vertices are $Ev$ and $Od$ respectively, and whose edges are $\{2v,2v+P\}$ and $\{2v+1,2v+1+Q\}$, respectively.  Clearly, $T_0$ and $T_1$ are $2$-regular and condition $(i)$ that $3P, 4P, 3Q,  4Q  \neq 0  $ means that the subgraphs $T_0$ and $T_1$ have girth at least five.
Now, observe that
the graph $G $ is an amalgamation $B(T_0, T_1)$ obtained by adding to $Ev$ all the edges of $T_0$ and by adding to $Od$ all the edges of $T_1$. Hence $G= S_{2t}( P, Q; k_1,   \ldots, k_w) $ is $(w+2)$-regular.
 Next, suppose that $C$ is a cycle in $G$ of size $3$ or $4$ which  must contain even and odd vertices.
If $C$ has a single even vertex, we have either   $k_i \pm Q -k_j = 0$ or $k_i \pm 2Q -k_j = 0$, depending on the size of~$C$, and both equalities contradict $(iii).$  If $C$ contains two even  and  two odd vertices, we have   $k_i \pm Q -k_j \pm P = 0$, again contradicting $(iii)$.  Therefore $G$ has girth at least $5$ iff conditions $(i), (ii),(iii)$ are satisfied.
\end{proof}

\noindent Notice that it is  useful to take $Q=2P$ because in this case  there are only four differences  $\pm   \{  P,    2P,  3P, 4P \}$ to be avoided.
Furthermore, if $S_{2t}(P,Q; k_1, \ldots, k_w)$  has girth $g \ge 5$, the (infinite) graph $S_{\infty} (P,Q; k_1, \ldots, k_w)$ also satisfies $g \ge 5$. We are interested in the converse result.

\begin{definition}
We call span $D$ of a graph $S_{\infty} (P,Q; k_1, \ldots, k_w)$ the maximum element of the set $\{|k_i|, k_i-k_j,\omega-\omega'\}$, with $\omega,  \omega' \in  \{0,\pm P,\pm Q \}$.
\end{definition}

\begin{lemma}  \label{lemmaspan}
Given even positive  $P \neq Q$ and odd different  $k_1, \ldots, k_w$ integers,  let us consider a graph $S_{\infty}(P,Q; k_1, \ldots, k_w)$ with  girth $g \geq 5$ and span $D$. If $t  \ge  D+1 $, then
\begin{itemize}
\item[$(i)$]  $0< P,Q, |k_i| < t$, so the graph $S_{2t}(P,Q; k_1, \ldots, k_w)$  has regularity $w+2$.
\item[$(ii)$] $S_{2t}(P,Q; k_1, \ldots, k_w)$ has girth at least $5$.
\end{itemize}
\end{lemma}

\begin{proof} By definition, $0 < P, Q \leq D$ and $-D \leq k_i \leq D$. As $ t \ge D+1$, item $(i)$ is immediate. Let also see that $S_{2t}(P,Q; k_1, \ldots, k_w)$  has girth $g \ge 5$.
Given two different pairs of odd weights, we have $k_i-k_j \neq k_p-k_q$ in $ \mathbb{Z}$. Also, from the definition of $D$, we have
$-D \leq  k_i-k_j,  k_p-k_q \leq D$ and hence, $-{2t} < (k_i-k_j)-( k_p-k_q) <{2t}$. So,  $k_i-k_j \neq k_p-k_q$ in $ \mathbb{Z}_{2t}$. The same argument shows $k_i-k_j \neq \omega-\omega'$ in $ \mathbb{Z}_{2t}$.
These are the  conditions $(ii), (iii)$ of Lemma~\ref{semicirculant}. Notice that  condition $(i)$ of the  Lemma~\ref{semicirculant} has been explicitly stated.
\end{proof}

\noindent As an example, let us mention that the  graph $S_{\infty}(2,4; 3,-7)$ has girth $5$ and span $D=10$.  Therefore, the graph $S_{2t}(2,4; 3,-7)$ is a 4-regular with girth $5$ for orders ${2t} \ge 22$.

\noindent In the next subsection  we construct two pairs  of regular graphs of girth~$5$ suitable for amalgamation into $L_q$ for some values of $q$.

\subsection {Elliptic semiplanes of type  {\emph{L}}}

Recall that the Levi Graph of an elliptic semiplane of type ${\emph{L}}$ is denoted by $L_q$. Following the terminology of Funk in \cite{F09} we say that two $r$-regular graphs $G_0$ and $G_1 $  with girth at least five are \emph{suitable for amalgamation into the elliptic semiplane  $L_q$} if  they are labeled with the elements of the cyclic group $(\mathbb{Z}_{q-1},+)$ with disjoint sets of Cayley colors in this group.
When~$q$ is odd, the fact that $\mathbb{Z}_{q-1}$ has $q-1$ elements suggests the use of this semiplane, because $r$-regular graphs with odd degree  have even order.

\noindent As in Section~3, the  amalgamation of a pair of $r$-regular suitable graphs into the  elliptic semiplane  $L_q$ gives a $(q+r, 5)$-graph $ \mathcal{L}_q (G_0, G_1).$ It has  $2 (q^2-1)$ vertices and deleting   pairs of blocks of vertices from $ \mathcal{L}_q (G_0, G_1) $, for regularities $k \leq q+r$, we have
\begin{equation}  \label{eq:delete2}
n(k,5) \leq 2 (q-1) (k-r+1).
\end{equation}

\noindent With $q=19$ vertices it is possible to construct a unique $4$-regular graph with girth $5$, the $(4,5)$-cage due to Robertson in~\cite{R64}. The use of the highest value of
$r \geq 4 $ for a given  $q>19$ increases the accuracy of the inequality (\ref{eq:delete2}).
Funk in~\cite{F09} constructs the best possible regular amalgams for $q \in \{  23,25,27 \}$ and hence we focus on primes $q \ge 29$.  Next, we give a construction of graphs which provide accurate amalgams for $q\in\{29,31, 37,41, 43,47\}$. 
\\

\noindent {\bf{Construction 2:}}

\begin{itemize}
\item For $q=29$:

Consider the graphs $G_0=S_{28} (4,8; 1,-1 )$   and $G_1=S_{28} (2,6; 3,-7)$  showed in Figure \ref{case28}. They are a suitable pair, that is,  both graphs are  $4$-regular, have girth five and have disjoint sets of Cayley colors, concretely  $\pm \{ 1,  4,  8\}$ and $\pm \{ 2, 3, 6,  7\},$ respectively. Hence, the $33$-regular graph $\mathcal{L}_{29} (G_0,G_1)$  has girth 5 and order 1680. We have generated the graph $\mathcal{L}_{29} (G_0,G_1)$ and observed that it has diameter $4$.
\begin{figure}[h!]
 \centering
\includegraphics[width=0.75\textwidth]{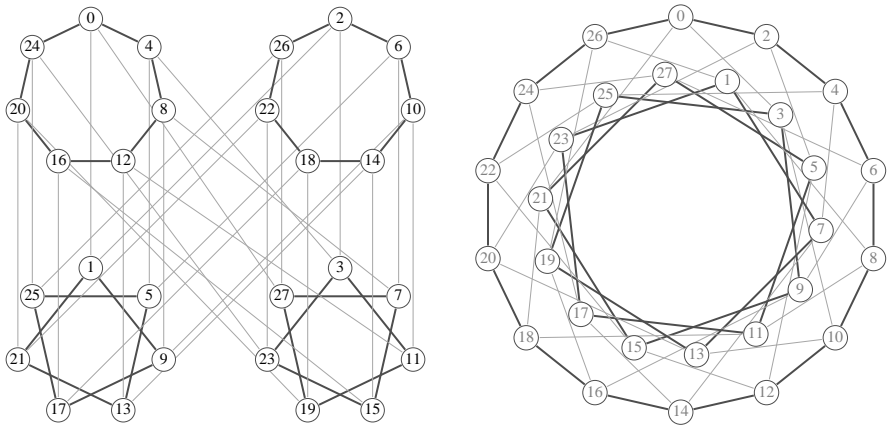}
\caption{\small   $S_{28} (4,8; 1,-1 )$   and $S_{28} (2,6; 3,-7)$, a pair of suitable graphs over $\mathbb{Z}_{28}.$ }
\label{case28}
\end{figure}
\noindent Deletion and inequality (\ref{eq:delete2})  provide entries $k= 32, 33$ of Table \ref{table3}.

\item For $q=31$:

There exist four $(5,5)$-cages (see \cite{M99, R69, YZ89,W73, W82} )  and the  graph $G_0=S_{30} (6, 12; 1,-1,  9)$ is isomorphic to the Foster one. The second suitable half $G_1$  has been found  with the following relabeling of the vertices.

{\scriptsize
\begin{center}
\begin{tabular}  {
  |  c    |  c    |   c       |  c      |   c     |  c  |   }
\hline
    $G_0  \    \leftrightarrow  \   G_1 $    &   $G_0  \    \leftrightarrow  \  G_1 $   &    $G_0  \    \leftrightarrow  \   G_1 $   & $G_0   \   \leftrightarrow  \  G_1 $ &    $G_0  \    \leftrightarrow  \   G_1 $       &   $G_0  \   \leftrightarrow  \   G_1 $   \\
       \hline \hline

$ \  0    \  \leftrightarrow \     0 \  $ &                1   $ \leftrightarrow$    28     &         2 $ \leftrightarrow$   1    &
 3  $ \leftrightarrow$   27   &         4     $ \leftrightarrow$     2 &            5  $ \leftrightarrow$    19  \\
\hline
$ \   6   \  \leftrightarrow  \    4 \   $ &               7  $ \leftrightarrow$    7       &         8 $ \leftrightarrow$   5  &
 9  $ \leftrightarrow$    22  &          10   $ \leftrightarrow$  6   &         11  $ \leftrightarrow$ 3\\
\hline
$  12    \  \leftrightarrow \   8 \  $   &               13  $ \leftrightarrow$     24      &         14 $ \leftrightarrow$   9   &
   15  $ \leftrightarrow$   20     &       16  $ \leftrightarrow$  10   &        17  $ \leftrightarrow$ 15 \\
\hline
$ 18   \  \leftrightarrow  \  12  $  &              19  $ \leftrightarrow$    23      &          20 $ \leftrightarrow$   13   &
  21  $ \leftrightarrow$  11    &          22  $ \leftrightarrow$ 14  &         23  $ \leftrightarrow$ 29 \\
\hline
$ 24    \  \leftrightarrow \  16   $ &              25  $ \leftrightarrow$     21      &          26 $ \leftrightarrow$   17   &
   27  $ \leftrightarrow$   25    &        28  $ \leftrightarrow$   18 &            29  $ \leftrightarrow$ 26 \\
       \hline
\end{tabular}
\end{center}}

\noindent Since the  Cayley colors of $G_1$ are the elements of the set $\mathbb{Z}_{30} - \{ 0,\pm1,\pm 6,\pm 9,  \pm 12\}$, the graphs $G_0$ and $G_1$ have disjoint Cayley colors, and therefore, the amalgam graph $ \mathcal{L}_{31} ( G_0,G_1)$ has girth $5$, regularity $36$ and order $2(31^2-1)=1920.$
Block deletion provides $n(35,5)  \leq 1860$ and $n(34,5) \leq 1800$.

\item For $q=37$:  

Consider the graphs $G_0=S_{36} (8, 14;  1,-1, 11)$ and $G_1=S_{36} (2,4;  3,-7,15)$   defined on the cyclic group
$( \mathbb{Z}_{36} , +)$. Both graphs are $5$-regular, have girth five and disjoint Cayley colors, concretely  $ \pm \{ 1,  8,  11,  14\}$ and $ \pm \{  2,  3, 4, 7, 15 \},$ respectively. Hence, the $42$-regular graph $\mathcal{L}_{37} (G_0,G_1)$  has girth 5 and order 2736.
Deletion provides $n(41,5)  \leq 2664$,  $n(40,5)  \leq 2592$,  $n(39,5)  \leq 2520$,  $n(38,5)  \leq 2448$.

\item For $q=41$:

The   $(6,5)$-cage is unique and it is well known (see \cite{KW79}) that   it   can be constructed by removing the vertices of a Petersen graph from the Hoffman-Singleton cage. We present a construction of the $(6,5)$-cage as the graph $S_{40}(8, 16; 1,-1,5,-13 )$. We denote it by $G_0$.
Due to the uniqueness of this cage, we construct a suitable graph $G_1$ according to the following relabeling of the vertices.

{\scriptsize
\begin{center}
\begin{tabular}  {
  |  c    |  c    |   c       |  c      |   c     |  c  |   c     |  c  |   }
\hline
    $G_0  \    \leftrightarrow  \   G_1 $    &   $G_0  \    \leftrightarrow  \  G_1 $   &    $G_0  \   \leftrightarrow  \   G_1 $   & $G_0   \   \leftrightarrow  \  G_1 $ &    $G_0  \    \leftrightarrow  \   G_1 $       &   $G_0  \   \leftrightarrow  \   G_1 $   &   $G_0  \   \leftrightarrow  \   G_1 $    &   $G_0  \   \leftrightarrow  \   G_1 $ \\
       \hline \hline

$ \  0    \  \leftrightarrow \   \  0$               &      $ \; \  1  \   \leftrightarrow \    12  $                 &        $\, \  2  \  \leftrightarrow \   \   1  $  &
$ \  3  \  \leftrightarrow  \  20  $                &      $  \  4   \  \leftrightarrow  \   \  2   $               &        $ \   5 \   \leftrightarrow   \   33   $  &
$ \  6  \  \leftrightarrow \  \    3   $             &      $  \ 7  \   \leftrightarrow   \  37  $  \\
\hline
$ \   8   \  \leftrightarrow \     \  7    $    &    $  \; \   9  \  \leftrightarrow  \   38    $      &      $   10 \   \leftrightarrow  \  \  8  $   &
$  11 \   \leftrightarrow   \  18 $                 &       $   12  \   \leftrightarrow  \  \  9   $           &      $   13 \   \leftrightarrow \  32  $  &
$14  \   \leftrightarrow  \   10  $                    &        $    15  \  \leftrightarrow  \   25 $\\
\hline
$  16    \  \leftrightarrow \  14  $   &       $        17  \  \leftrightarrow   \ \  5  $         &        $ 18  \  \leftrightarrow   \ 15  $  &
$   19 \   \leftrightarrow  \  13    $          &       $ 20 \   \leftrightarrow  \  16           $        &        $ 21 \   \leftrightarrow   \ 36    $  &
$  22  \  \leftrightarrow  \   17 $             &        $ 23 \   \leftrightarrow    \  27   $   \\
\hline
$ 24   \  \leftrightarrow \    21  $  &          $    25 \  \leftrightarrow   \  19   $  &         $ 26  \  \leftrightarrow  \  22  $ &
 $ 27  \   \leftrightarrow  \ 11  $  &          $28  \   \leftrightarrow \  23 $                &         $29  \  \leftrightarrow \  35   $  &
$   30  \  \leftrightarrow    \  24   $      &          $  31  \   \leftrightarrow    \ 39 $\\
\hline
$ 32    \  \leftrightarrow \  28   $ &           $  33  \  \leftrightarrow    \  26    $  &         $ 34  \  \leftrightarrow   \ 29 $  &
 $  35  \  \leftrightarrow \ \   4   $ &       $ 36  \  \leftrightarrow   \ 30 $  &        $    37  \  \leftrightarrow \ 34   $  &
  $  38   \  \leftrightarrow \ 31$       &            $39   \  \leftrightarrow  \ \   6  $\\
       \hline
\end{tabular}
\end{center}}

 \noindent Since  $G_0$ and $G_1$ have no Cayley color in common, the $47$-regular graph $\mathcal{L}_{41} (G_0,G_1)$  has girth 5 and order $2(41^2-1)=3360$.
Deletion   and inequality (\ref{eq:delete2})  provide entries $k= 43 \ldots 46$ of Table~\ref{table3}.

\item For $q=43, 47$:

Since we resort to $5$-regularity, there exist several pairs of suitable graphs. For the sake of uniformity, we consider the graphs $S_{q-1} (6,12;1,-1,9)$ and $S_{q-1} (2, 4;3,-7, 15).$  It proves that  $n(48,5) \le2(43^2-1)= 3696$ and  $n(52,5) \le2(47^2-1)= 4416$.
As a curiosity, let us mention that the graph $S_{42}(1,-1,-7,11,15)$ is the $(5,6)$-cage and it forms a suitable pair with  $S_{42}(2,4; 5,-5,17)$.

\end{itemize}

\noindent Based on the above constructions  and recalling that it is possible to delete blocks of points and lines  we can write the following theorem.

\begin{theorem}
The following upper bound on the order  of a $k$-regular graph of girth $5$ holds
{
\begin{center}
\begin{tabular}   {  | c    || c  |  }
\hline
 $k$ &   $rec(k,5)$     \\
       \hline \hline
       $32, 33$ & $56 (k-3)$ \\ \hline
  $   34,35,36 $   &    $60 (k-4) $   \\
  \hline
   $  38, \ldots, 42   $&  $  72 (k-4)   $ \\
       \hline
       $43, \ldots, 47$ & 80 (k-5)  \\ \hline
        $ 48   $&  $  3696  $ \\
       \hline
  $  49, \ldots, 52   $&  $  92 (k-4)   $ \\
       \hline
\end{tabular}
\end{center}}

\end{theorem}

\noindent To finalize this section we prove Theorem \ref{th11}. In this case we generate a pair of  6-regular suitable graphs to be amalgamated into $L_q$, for  an odd prime power $q \ge 49$. We start with $q=49$; notice that this case is sharp because the Hoffman-Singleton graph is the cage that attains the lower bound $n_0(7,5)=50$ (see~\cite{HS60})

\noindent{\bf{Theorem 1.1}}  {\it{Given an integer $k \ge 53$, let $q$ be the lowest odd prime power, such that  $k \leq q+6$. Then  $n(k,5) \leq 2(q-1)(k-5)$.}}

\begin{proof}
First consider $q=49$. Add to the $4$-regular bipartite graph $ S_{48}(1,-1,5,-13)$ the edges $\{2v, 2v+8 \}$ over the even vertices of $\mathbb{Z}_{48}$, and the four cycles $ \{1+i, 17+i, 41+i, 25+i, 9+i, \\   33+~i, 1+i \}$, for $i=0,2,4,6$,  over the odd vertices.  We call $G_0$ to this $(6,5)$-graph.
 To construct a suitable graph $G_1$, we resort to the following relabeling of the vertices

{\scriptsize
\begin{center}
\begin{tabular}  {
  |  c    |  c    |   c       |  c      |   c     |  c  |   c     |  c  |   }
\hline
    $G_0  \    \leftrightarrow  \   G_1 $    &   $G_0  \    \leftrightarrow  \  G_1 $   &    $G_0  \   \leftrightarrow  \   G_1 $   & $G_0   \   \leftrightarrow  \  G_1 $ &    $G_0  \    \leftrightarrow  \   G_1 $       &   $G_0  \   \leftrightarrow  \   G_1 $   &   $G_0  \   \leftrightarrow  \   G_1 $    &   $G_0  \   \leftrightarrow  \   G_1 $ \\
       \hline \hline

$ \  0    \  \leftrightarrow \   \  0$   &      $ \; \  1  \   \leftrightarrow \    42  $   &      $\, \  2  \  \leftrightarrow \ \   1  $  &
$ \  3  \  \leftrightarrow  \  39  $     &      $  \  4   \  \leftrightarrow  \  \  2   $   &      $\, \ 5 \ \leftrightarrow   \   23   $  &
$ \  6  \  \leftrightarrow \  \  3 $   &      $  \ 7  \   \leftrightarrow   \  47  $  \\
\hline
$ \ 8 \  \leftrightarrow \ \  6  $     &    $  \; \   9  \  \leftrightarrow  \ \ 4  $     &      $   10 \   \leftrightarrow  \  \  7  $   &
$ 11 \   \leftrightarrow   \  28 $     &    $   12  \   \leftrightarrow  \  \  8   $       &      $   13 \   \leftrightarrow \  34  $  &
$ 14  \   \leftrightarrow \ \ 9  $   &    $    15  \  \leftrightarrow  \   43 $\\
\hline
$ 16    \  \leftrightarrow \  12  $    &    $ 17  \  \leftrightarrow   \ 35  $         &       $ 18  \  \leftrightarrow   \ 13  $  &
$ 19 \   \leftrightarrow  \  36  $     &    $ 20 \   \leftrightarrow  \  14      $         &      $ 21 \   \leftrightarrow   \ 29 $  &
$ 22  \  \leftrightarrow  \   15 $     &    $ 23 \   \leftrightarrow    \ 44   $   \\
\hline
$ 24   \  \leftrightarrow \ 18  $      &    $ 25 \  \leftrightarrow  \  37 $            &      $ 26  \  \leftrightarrow  \  19  $ &
$ 27  \   \leftrightarrow  \ \ 5  $     &    $ 28  \   \leftrightarrow \  20 $               &      $29  \  \leftrightarrow \  40   $  &
$ 30  \ \leftrightarrow  \  21   $     &    $  31  \   \leftrightarrow    \ 10 $\\
\hline
$ 32  \  \leftrightarrow \  24   $     &    $ 33  \  \leftrightarrow    \  45   $       &      $ 34  \  \leftrightarrow   \ 25 $  &
$ 35  \  \leftrightarrow \ 46 $     &    $ 36  \  \leftrightarrow   \ 26 $            &      $ 37  \  \leftrightarrow \ 38   $  &
$ 38  \  \leftrightarrow \ 27$        &    $ 39   \  \leftrightarrow  \ 16  $\\
\hline
$ 40  \  \leftrightarrow \  30   $    &       $ 41  \  \leftrightarrow  \ 41    $       &      $ 42  \  \leftrightarrow   \ 31 $  &
$ 43  \  \leftrightarrow \ 17   $      &       $ 44  \  \leftrightarrow   \ 32 $            &      $ 45  \  \leftrightarrow \ 11   $  &
$ 46  \  \leftrightarrow \ 33$        &       $ 47   \  \leftrightarrow  \ 22  $\\
\hline
\end{tabular}
\end{center}}
\noindent The graphs $G_0$ and $G_1$ have disjoint Cayley colors, namely $w(G_0)=\pm \{  1,  5, 8, 13, 16, 24 \}$ and
 $w(G_1)=\mathbb{Z}_{48} -( w(G_0)   \cup   \{0 \} )$.
Hence, $G_0$ and $G_1$ is a suitable pair of graphs for amalgamation into $L_{49}$.  Using these graphs and also the fact that we can delete blocks of points and lines we prove the theorem for $53\leq k \leq 55$.

\noindent When $q \in \{ 53, 67, 71, 79, \ldots\} $ is an odd prime power, we consider the $6$-regular graphs $G_0= S_{q-1}( 8,16; 1, -1, 5,-13 )$  and
$G_1 = S_{q-1}( 2,4; 3, -7, 15, -21 )$. Direct checking shows their suitability over $L_q$ for $q=53,67,71$.  When $q  \ge  79$,  the suitability of $G_0$ and~$G_1$ is a consequence of  Lemma~\ref{lemmaspan}, because  the infinite graphs  $ S_{\infty}( 8,16; 1, -1, 5,-13 )$ and $ S_{\infty}( 2,4; 3, -7, 15, -21 )$ have girth $5$ and spans  $32$ and $37$, respectively.
When \emph{$q \in \{ 59, 61, 73 \} $},  the graph $G_0= S_{q-1}( 8,16; 1, -1, 5,-13 )$ combined with
$G_1 = S_{q-1}( 2,4; 3, -7, 15, \alpha )$, where $\alpha=-23 $ for $q=59, 73$ and $\alpha=-25 $ for $q=61$, is a suitable pair of graphs over $L_q.$
Therefore, for $ q \ge 49$, the $(q+6)$-regular graph $\mathcal{L}_{q} (G_0,G_1)$  has girth at least~$5$  and order $2 (q^2-1) $.   Also, according to inequality~(\ref{eq:delete2}),  $n(k,5)  \le  2(q-1)  (k-5)$, for regularities $ 56 \le k \le q+6.$
 \end{proof}

\noindent Notice that Theorem~\ref{th11} improves J{\o}rgensen's result $n(q+\lfloor{\frac{\sqrt{q-1}}{4}}\rfloor,5)\leq 2(q^2-1)$ (see \cite{J05}) for $k \leq 577$ and ties with it for $ 578 \leq k \leq 779$.

 \section{General constructions for $q=2^m$.}

 In this section we work with the same ideas used in the two previous sections. We amalgamate into $C_q$ for $q=2^m$ when $m\geq 5$ applying Theorem \ref{teoC} on regular graphs. The case $m=4$ is considered in Section 3, where we amalgamate bi-regular graphs.

\noindent First, we deal with $m=5$ or $q=32$. Since an $r$-regular graph with $32$ vertices and girth $5$ can reach at most $5$-regularity, we have the following sharp result:

\begin{theorem}  \label{case32}
There exists a $37$-regular graph with girth $5$ and order $2048.$
\end{theorem}

\begin{proof}
As in the case $q=16$,  denote the elements of  $(\mathbb{F}_{32},+) \cong  ( (\mathbb{Z}_2)^5,+)$ by  $defgh$ instead of $ \{d,e,f,g,h \}$.
Let $G_0$ be the $(5,5)$-graph with order $32$ and with the following adjacency list:

{\scriptsize
\begin{center}
\begin{tabular}  {
  |  c    |  c    |   c       |  c      |   c     |  c  |   }
\hline
    Vertex    &   Adjacent vertices  &    & Vertex  &   Adjacent vertices         \\
       \hline \hline
 00000  & 10000,  11010,  11100, 00001, 11111   & &   00001  & 00000,  10001,  11011,  11101,  11110 \\
 10000  & 00000,  01011,  01101,  01110,  11001  & &  10001 & 00001, 01010, 01100, 01111, 11000  \\
 01000  & 01001, 10010, 10101, 10110, 11000      & &  01001  &  01000, 10011, 10100, 10111, 11001 \\
11000   & 00011, 00100, 00111, 01000, 10001      & &  11001  & 00010, 00101, 00110, 01001, 10000  \\
00100   & 00101, 10100, 11000, 11010, 11110      & &   00101  & 00100, 10101, 11001, 11011, 11111 \\
10100   & 00100, 01001, 01011, 01111, 11101     & &  10101 & 00101, 01000, 01010, 01110, 11100  \\
01100  & 01101, 10001, 10011, 10110, 11100      & &  01101  &  01100, 10000, 10010, 10111, 11101 \\
11100   & 00000, 00010, 00111, 01100, 10101     & &  11101  & 00001, 00011, 00110, 01101, 10100  \\
00010  & 00011, 10010, 11001, 11100, 11110      & &   00011  & 00010, 10011, 11000, 11101, 11111 \\
10010  & 00010, 01000, 01101, 01111, 11011      & &  10011 & 00011, 01001, 01100, 01110, 11010  \\
01010  & 01011, 10001, 10101, 10111, 11010      & &  01011  &  01010, 10000, 10100, 10110, 11011 \\
11010   & 00000, 00100, 00110, 01010, 10011      & &  11011  & 00001, 00101, 00111, 01011, 10010  \\
00110   & 00111, 10110, 11001, 11010, 11101      & &   00111  & 00110, 10111, 11000, 11011, 11100 \\
10110   & 00110, 01000, 01011, 01100, 11111     & &  10111 & 00111, 01001, 01010, 01101, 11110  \\
01110  & 01111, 10000, 10011, 10101, 11110      & &  01111  &  01110, 10001, 10010, 10100, 11111 \\
11110  & 00001, 00010, 00100, 01110, 10111      & &  11111  &  00000, 00011, 00101, 01111, 10110 \\
\hline
\end{tabular}

\end{center}}

\noindent The set
 $ w(G_0)= \{ 00001, 01001,10000, 11010, 11011,  11100, 11101,  11110,  11111\}$ contains the  Cayley colors  of $G_0$.
As graph $G_1$,  consider  the isomorphic graph of $G_0$ with the following relabeling of the vertices:

{\scriptsize
\begin{center}
\begin{tabular}  {
  |  c    |  c    |   c       |  c      |   c     |  c  |   }
\hline
    $G_0  \ \   \leftrightarrow  \ \  G_1 $    &   $G_0  \ \   \leftrightarrow  \ \  G_1 $   &    $G_0  \ \   \leftrightarrow  \ \  G_1 $   & $G_0  \ \   \leftrightarrow  \ \  G_1 $ &    $G_0  \ \   \leftrightarrow  \ \  G_1 $       &   $G_0  \ \   \leftrightarrow  \ \  G_1 $   \\
       \hline \hline

00000  $ \leftrightarrow$   00000  &                00001   $ \leftrightarrow$    00011     &         00010 $ \leftrightarrow$   00010    &
 00011  $ \leftrightarrow$   00001   &         00100     $ \leftrightarrow$     00100 &            00101  $ \leftrightarrow$    00111\\
\hline
00110    $ \leftrightarrow$  00110  &               00111  $ \leftrightarrow$    01110       &         01000 $ \leftrightarrow$   11001  &
 01001  $ \leftrightarrow$    11100  &          01010   $ \leftrightarrow$  11111   &         01011  $ \leftrightarrow$ 11011\\
\hline
01100    $ \leftrightarrow$ 10011   &               01101  $ \leftrightarrow$     11101      &         01110 $ \leftrightarrow$   11010   &
   01111  $ \leftrightarrow$   11110     &       10000  $ \leftrightarrow$  01111   &        10001  $ \leftrightarrow$ 10100 \\
\hline
10010  $ \leftrightarrow$  01100    &              10011  $ \leftrightarrow$    10000      &          10100 $ \leftrightarrow$   01000   &
  10101  $ \leftrightarrow$  10001    &          10110  $ \leftrightarrow$ 01010  &         10111  $ \leftrightarrow$ 11000 \\
\hline
11000   $ \leftrightarrow$ 10110    &              11001  $ \leftrightarrow$     01101      &          11010 $ \leftrightarrow$   10101   &
   11011  $ \leftrightarrow$   01001    &        11100  $ \leftrightarrow$   00101 &            11101  $ \leftrightarrow$ 01011 \\
\hline
11110   $ \leftrightarrow$  10111    &               11111  $ \leftrightarrow$  10010         &                                   &              &           &  \\
       \hline
\end{tabular}
\end{center}}

\noindent Since the set of  Cayley colors of $G_1$ is $w(G_1)=\mathbb{F}_{32} - (w(G_0) \cup  \{0000,00110 \}  )$, the graphs $G_0$ and $G_1$ have disjoint Cayley colors, and therefore, the amalgam graph $ \mathcal{C}_{32} ( G_0,G_1)$ has girth $5$, regularity $37$ and order $2 \cdot 32^2=2048.$
\end{proof}

\noindent To give a general result for $m\geq 6$ we need some equivalences and definitions. As usual we identify the elements of $\mathbb{F}_{2^m} \cong  (\mathbb{Z}_2)^m $ with a number of $\mathbb{Z}_{2^m} $ in the following way:
$$(v_{m-1}, \ldots, v_{0} ) \; \longleftrightarrow  \; \displaystyle \sum_{ i =0} ^{ m-1} 2^i v_i$$
for every $i=0, \ldots, m-1$ and  $v_i \in \mathbb{Z}_2.$
This induces a bijection $\phi:  \mathbb{Z}_{2^m}  \rightarrow (\mathbb{Z}_{2} )^m  $ such that the elements of  $(\mathbb{Z}_{2} )^m$ can be represented either by a vector or by a number.

\noindent This bijective relationship allows to translate the graph $ S_{2^m} (P,Q; k_1, \ldots, k_w)$ with vertex set $\mathbb{Z}_{2^m}$ into a new graph with vertex set $(\mathbb{Z}_{2} )^m$ defined as follows:

\begin{definition}
Given an integer $N =2^m $,   a sequence $k_1,  \ldots, k_w$  of different odd elements from  $\mathbb{Z}_N$ and two even elements $0 < P, Q  <  N / 2  $,
we denote by $ \bar{S}_{2^m} (P,Q; k_1, \ldots, k_w)$ the graph with vertex set $(\mathbb{Z}_{2} )^m$ obtained by translating  the vertices and edges of $ S_{2^m} (P,Q; k_1, \ldots, k_w)$  by means of the bijection  $\phi: \mathbb{Z}_{2^m}   \rightarrow (\mathbb{Z}_{2} )^m.$
\end{definition}

\noindent Clearly, graphs $ S_{2^m} (P,Q; k_1, \ldots, k_w)$ and $ \bar{S}_{2^m} (P,Q; k_1, \ldots, k_w)$ are isomorphic. Notice that the Cayley colors of the graph $ \bar{S}_{2^m} (P,Q; k_1, \ldots, k_w)$ are  computed in the additive group $ (\mathbb{Z}_{2} )^m$; which implies that edges of $ \bar{S}_{2^m} (P,Q; k_1, \ldots, k_w)$ associated to an element of $\{P,Q; k_1, \ldots, k_w\}$ might have different Cayley colors in $ (\mathbb{Z}_{2} )^m$.

\noindent To finish this section we prove Theorem \ref{th12}, in which we consider even primes $q \ge 64$ and construct a pair of suitable $6$-regular graphs whose amalgamation into $C_{q}$  establishes a general bound on $n(k,5)$ for regularities $68 \leq k \leq q+6.$

\noindent {\bf{Theorem 1.2}}  {\it {Given an integer $k \ge 68$, let $q=2^m$ be the lowest even prime, such that $k \leq q+6$. Then  $n(k,5) \leq 2 q (k-6)$.}}

\begin{proof}
Consider $q=2^m$ for an integer $m \ge 6$.  Due to the bijection $\phi$ described above we represent the elements of $(\mathbb{Z}_2)^m $ by the numbers of $\mathbb{Z}_{2^m}$ and vice versa.

\noindent For $q=64$ we consider the $6$-regular graph $G_0= \bar{S}_{64} ( 4,8; 1,3, 41,47) $ of girth five  and set of Cayley colors  $w(G_0)=\{1, 3, 4, 7, 8, 12, 15, 19, 23, 24, 25, 28, 31, 41, 47, 51, 55, 56, 57, 60, 63\} .$ To obtain the graph $G_1$ we consider the following relabeling of the vertices:

{\scriptsize
\begin{center}
\begin{tabular}  {
  |  c    |  c    |   c       |  c      |   c     |  c  |   c     |  c  |   }
\hline
    $G_0  \    \leftrightarrow  \   G_1 $    &   $G_0  \    \leftrightarrow  \  G_1 $   &    $G_0  \   \leftrightarrow  \   G_1 $   & $G_0   \   \leftrightarrow  \  G_1 $ &    $G_0  \    \leftrightarrow  \   G_1 $       &   $G_0  \   \leftrightarrow  \   G_1 $   &   $G_0  \   \leftrightarrow  \   G_1 $    &   $G_0  \   \leftrightarrow  \   G_1 $ \\
       \hline \hline

$ \  0    \  \leftrightarrow \   \  0$   &      $ \; \  1  \   \leftrightarrow \    44  $   &      $\, \  2  \  \leftrightarrow \ \   2  $  &
$ \  3  \  \leftrightarrow  \  39  $     &      $  \  4   \  \leftrightarrow  \  \  5   $   &      $\, \ 5 \ \leftrightarrow   \   41   $  &
$ \  6  \  \leftrightarrow \  \  7 $   &      $  \ 7  \   \leftrightarrow   \  19  $  \\
\hline
$ \ 8 \  \leftrightarrow \  12  $     &    $ \; \    9  \  \leftrightarrow  \  50  $     &      $   10 \   \leftrightarrow  \   14  $   &
$ 11 \   \leftrightarrow   \  28 $     &    $   12  \   \leftrightarrow  \  \  1   $       &      $   13 \   \leftrightarrow \  52  $  &
$ 14  \   \leftrightarrow \ \ 3  $   &    $    15  \  \leftrightarrow  \   21 $\\
 \hline
$ 16    \  \leftrightarrow \ \ 4  $    &    $ 17  \  \leftrightarrow  \ 25  $         &       $ 18  \  \leftrightarrow   \ \ 6  $  &
$ 19 \   \leftrightarrow  \  22  $     &    $ 20 \   \leftrightarrow  \  57      $         &      $ 21 \   \leftrightarrow   \ 20 $  &
$ 22  \  \leftrightarrow  \  59 $     &    $ 23 \   \leftrightarrow    \ 31   $   \\
\hline
$ 24   \  \leftrightarrow \ 24  $      &    $ 25 \  \leftrightarrow  \  45 $            &      $ 26  \  \leftrightarrow  \ 26  $ &
$ 27  \   \leftrightarrow  \ 56  $     &    $ 28  \   \leftrightarrow \ 61 $               &      $29  \  \leftrightarrow \ 48   $  &
$ 30  \ \leftrightarrow  \ 63   $     &    $  31  \   \leftrightarrow \ 29 $\\
\hline
$ 32  \  \leftrightarrow \ 32   $     &    $ 33  \  \leftrightarrow    \ 10   $       &      $ 34  \  \leftrightarrow   \ 34 $  &
$ 35  \  \leftrightarrow \ \ 8 $     &    $ 36  \  \leftrightarrow   \ 49 $            &      $ 37  \  \leftrightarrow \ 23   $  &
$ 38  \  \leftrightarrow \ 51 $        &    $ 39   \  \leftrightarrow  \ 27  $\\
\hline
$ 40  \  \leftrightarrow \  36   $    &       $ 41  \  \leftrightarrow  \ 62    $       &      $ 42  \  \leftrightarrow   \ 38 $  &
$ 43  \  \leftrightarrow \ 54   $      &       $ 44  \  \leftrightarrow   \ \ 9 $            &      $ 45  \  \leftrightarrow \ 35   $  &
$ 46  \  \leftrightarrow \ 11$        &       $ 47   \  \leftrightarrow  \ 43  $\\
\hline
$ 48  \  \leftrightarrow \ 40   $    &       $ 49  \  \leftrightarrow  \ 46    $       &      $ 50  \  \leftrightarrow   \ 42 $  &
$ 51  \  \leftrightarrow \ 30   $      &     $ 52  \  \leftrightarrow   \ 53 $            &      $ 53  \  \leftrightarrow \ 33   $  &
$ 54  \  \leftrightarrow \ 55$        &       $ 55   \  \leftrightarrow  \ 17  $\\
\hline
$ 56  \  \leftrightarrow \ 16   $    &       $ 57  \  \leftrightarrow  \ 58    $       &      $ 58  \  \leftrightarrow   \ 18 $  &
$ 59  \  \leftrightarrow \ 60   $      &     $ 60  \  \leftrightarrow  \ 13 $            &      $ 61  \  \leftrightarrow \ 47   $  &
$ 62  \  \leftrightarrow \ 15$        &       $ 63   \  \leftrightarrow \ 37 $\\
\hline
\end{tabular}
\end{center}}

\noindent   The Cayley colors of $G_1$ are $w(G_1)= \{1,\ldots, 63\}- w(G_0)- \{50\}$ and hence the $(70,5)$-graph $\mathcal{C}_{64} (G_0,G_1)$ has order $2 \cdot 64^2$.

 \noindent In general for $q=2^m$ and $m\geq 7$ we use the previous graphs $G_0$ and $G_1$ defined over $(\mathbb{Z}_2)^6 $ to construct new graphs $G^m_0$ and $G^m_1$ with vertex set $(\mathbb{Z}_2)^m$ in the following way:
The neighbors of a vertex  $(u_{m-1}, \ldots, u_{0})$ in $G^m_0$ are the six  vertices of the set $\{ (u_{m-1},...,u_6,v_5, \ldots, v_0 ) :   (u_5, \ldots ,u_0) (v_5, \ldots ,v_0)  \in E(G_0) \}.$  Similar definition holds for $G^m_1$. Graphs $G^m_0$ and $G^m_1$ are  formed by $2^{m-6}$ disconnected copies of $G_0$ and $G_1$, respectively, and therefore, both graphs are $6$-regular with girth $5$. Also, the sets of Cayley colors $w(G^m_0)= \{ (0,\ldots,0,\alpha_5, \ldots, \alpha_0) \in (\mathbb{Z}_2)^m : (\alpha_5, \ldots, \alpha_0)  \in w(G_0) \}$ and $w(G^m_2)= \{ (0,\ldots,0,\beta_5, \ldots, \beta_0) \in (\mathbb{Z}_2)^m : (\beta_5, \ldots, \beta_0) \in w(G_1) \}$ are disjoint because $w(G_0) \cap w(G_1)= \emptyset$.
Clearly, the graphs $G^m_0$ and $G^m_1$ are suitable for amalgamation into $\mathcal{C}_q$ and the  graph $\mathcal{C}_{q} (G^m_0,G^m_1)$ has regularity~$q+6$, order~$ 2 q^2 $ and girth at least five. For $k\leq q+6$ removing $q+6-k$ blocks of points and $q+6-k$ blocks of  lines we obtain a graph of order $2q^2-2q(q+6-k)$ and consequently $n(k,5) \leq 2 q (k-6)$.
\end{proof}

\noindent Clearly in this paper we improve $rec(k,5)$ for many values of $k$. As we mention at the end of Section~4 our Theorem \ref{th11} improves J{\o}rgensen's result for $k \leq 577$.
We consider that an interesting future work would be to extend Theorem~\ref{th11} to large odd prime powers and  to improve our Theorem \ref{th12} when $q$ is a power of two.

\subsection*{Acknowledgment}

Research supported by the Ministry of Education and Science,
Spain, the European Regional Development Fund (ERDF)
under project MTM2014-60127-P,   CONACyT-M\'exico under projects 178395, 166306, and  PAPIIT-M\'exico under project IN104915.


\end{document}